\documentclass[letterpaper, 10pt, conference]{ieeeconf}
\usepackage{amsmath,amssymb,url, float}
\usepackage{commath, enumerate}
\usepackage{graphicx,subfigure}
\usepackage[usenames,dvipsnames]{color}
\usepackage[pdfpagemode=none,pdfstartview=FitH]{hyperref}

\newcommand{\al}{\alpha}
\newcommand{\be}{\beta}

\newcommand{\Bd}{B_{\W_d}}
\newcommand{\diag}{\text{diag}}
\newcommand{\I}{I_{3\times3}}
\newcommand{\lam}{\lambda}
\renewcommand{\i}{\mathfrak{i}}

\renewcommand{\t}{\times}
\newcommand{\W}{{\Omega}}
\newcommand{\tr}{\text{tr}}

\newcommand{\tcb}{\textcolor{black}}

\newcommand{\no}{\nonumber}
\newcommand{\tU}{\tilde{U}}
\newcommand{\tG}{\tilde{G}}
\newcommand{\tg}{\tilde{g}}

\newcommand{\tih}{\tilde{h}}
\newcommand{\V}{\mathcal{V}}

\newcommand{\braces}[1]{\ensuremath{\left\{ #1 \right\}}}

\newcommand{\refeqn}[1]{(\ref{eq:#1})}

\newcommand{\SO}{\ensuremath{\mathsf{SO(3)}}}
\newcommand{\T}{\ensuremath{\mathsf{T}}}

\newcommand{\so}{\ensuremath{\mathfrak{so}(3)}}
\newcommand{\SE}{\ensuremath{\mathsf{SE(3)}}}

\renewcommand{\Re}{\ensuremath{\mathbb{R}}}

\newcommand{\Sph}{\ensuremath{\mathsf{S}}}

\title{\LARGE \bf
Spacecraft Position and Attitude Formation Control\\ using Line-of-Sight Observations}

\author{Tse-Huai Wu and Taeyoung Lee\authorrefmark{1}%
\thanks{Tse-Huai Wu and Taeyoung Lee, Mechanical and Aerospace Engineering, The George Washington University, Washington DC 20052. {\tt \{wu52,tylee\}@gwu.edu}}%
\thanks{\textsuperscript{\footnotesize\ensuremath{*}}This research has been supported in part by NSF under the grant CMMI-1243000 (transferred from 1029551), CMMI-1335008, and CNS-1337722.}
}

\newtheorem{lem}{Lemma}
\newtheorem{prop}{Proposition}
\newtheorem{assump}{Assumption}

\begin{document}
\allowdisplaybreaks
\maketitle \thispagestyle{empty} \pagestyle{empty}

\begin{abstract}
This paper studies formation control of an arbitrary number of spacecraft based on a serial network structure. The leader controls its absolute position and absolute attitude with respect to an inertial frame, and the followers control its relative position and attitude with respect to another spacecraft assigned by the serial network. The unique feature is that both the absolute attitude and the relative attitude control systems are developed directly in terms of the line-of-sight observations between spacecraft, without need for estimating the full absolute and relative attitudes, to improve accuracy and efficiency. Control systems are developed on the nonlinear configuration manifold, guaranteeing  exponential stability. Numerical examples are presented to illustrate the desirable properties of the proposed control system.
\end{abstract}

\section{Introduction}

Spacecraft formation flight has been intensively studied as distributing tasks over a group of low-cost spacecraft is more efficient and robust than operating a single large and powerful spacecraft~\cite{SchHad04}. For cooperative spacecraft missions, precise control of relative configurations among spacecraft is critical for success. For interferometer missions like Darwin, spacecraft in formation should maintain specific relative position and relative attitude configurations precisely. Precise relative position control and estimation have been addressed successfully, for example, by utilizing carrier-phase differential GPS~\cite{Mit04,GarChaPIAC05}.

For relative attitude control, there have been various approaches, including leader-follower strategy~\cite{KanYehIJRNC02,chou2011}, behavior-based controls~\cite{BeaLawITCST01,AbdTayITAC09} and virtual structures~\cite{RenBeaPAGNCC02, Ren2004}. These approaches have distinct features, but there is a common framework: the absolute attitude of each spacecraft in formation is determined independently and individually by using an on-board sensor, such as inertial measurement units and star trackers, and they are transmitted to other spacecraft to determine relative attitude between them. As the relative attitudes are determined \textit{indirectly} by comparing absolute attitudes, there is a fundamental limitation in accuracies. More explicitly, measurement and estimation errors of multiple sensors are accumulated in determination of the relative attitudes.

Vision-based sensors have been widely applied for navigation of autonomous vehicles, and recently, they are proposed for determination of relative attitudes. It is shown that the line-of-sight (LOS) measurements between two spacecraft determine the relative attitude between them completely, and based on it, an extended Kalman filter is developed~\cite{KimCraJGCD07,AndCraJGCD09}. Recently, these are also utilized in stabilization of relative attitude between two spacecraft~\cite{LeePACC12}, and tracking control of relative attitude formation between multiple spacecraft~\cite{WuACC13,Wu2012}, where control inputs are directly expressed in terms of line-of-sight measurements, without need for constructing the full, absolute attitude or the relative attitudes. 

However, these prior results are restrictive in the sense that the relative positions among spacecraft, and therefore the lines-of-sight with respect to the inertial frame, are assumed to be fixed during the whole attitude maneuvers. Therefore, they cannot be applied to the cases where both the relative positions and the relative attitudes should be controlled concurrently at the similar time scale. 

The objective of this paper is to eliminate such restrictions. In this paper, the translational dynamics and the rotational dynamics of each spacecraft are considered, and a serial network structure is defined. The first spacecraft at the network, namely the leader controls its absolute position and absolute attitude with respect to an inertial frame, and the remaining spacecraft, namely followers control its relative position and relative attitude with respect to another spacecraft ahead in the serial chain of network. The main contribution is that both the absolute attitude controller of the leader, and the relative attitude controller of the followers are defined directly in terms of the line-of-sight measurements, and exponential stability is guaranteed without the restrictive assumption that the relative positions are fixed.

Therefore, the control system proposed in this paper inherits the desirable features of the relative attitude controls based on the lines-of-sight~\cite{LeePACC12,WuACC13,Wu2012}, namely low cost and long-term stability, while requiring no corrections in measurements as opposed to gyros, or eliminating needs for computationally expensive star tracking algorithms. But, it can be applied to more realistic cases where the relative positions are controlled simultaneously. Another distinct feature of this paper is that the absolute attitude of the leader is also controlled, whereas the preliminary works are only focused on relative attitude formation control~\cite{LeePACC12}. All of these are constructed on the special orthogonal group to avoid singularities, complexities associated with local parameterizations or ambiguities of quaternions in representing attitudes.

\section{Problem Formulation}\label{sec:PF}

Consider $n$ spacecraft, where each spacecraft is modeled as a rigid body. Define an inertial frame, and the body-fixed frame for each spacecraft. The configuration of the $i$-th spacecraft is defined by $(R_i,x_i)\in\SE$, where the special Euclidean group $\SE$ is the semi-direct product of the special orthogonal group $\SO=\{R\in\Re^{3\times 3}\,|\, R^TR=I,\; \mathrm{det}[R]=1\}$, and $\Re^3$. The rotation matrix $R_i\in\SO$ represents the linear transformation of the representation of a vector from the $i$-th body fixed frame to the inertial frame, and the vector $x_i\in\Re^3$ denote the location of the mass center of the $i$-th spacecraft with respect to the inertial frame. 

\subsection{Dynamic Model}

Let $m_i\in\Re$ and $J_i\in\Re^{3\times 3}$ be the mass and the inertia matrix of the $i$-th spacecraft. The equations of motion are given by
\begin{gather}
m_i\ddot{x}_i=f_i,\label{eq:f} \\
\dot{x}_i=v_i,\\
J_i\dot\W_i + \W_i\times J_i\W_\i = u_i,\label{eq:Wdot}\\
\dot{R}_i = R_i\hat\W_i,\label{eq:Rdot}
\end{gather}
where $v_i,\Omega_i\in\Re^3$ are the translational velocity and the rotational angular velocity of the $i$-th spacecraft, respectively. The control force acting on the $i$-th spacecraft is denoted by $f_i\in\Re^3$, and the $i$-th control moment is denoted by $M_i\in\Re^3$. The vectors $v_i,f_i$ are represented with respect to the inertial frame, and $\Omega_i,M_i$ are represented with respect to the $i$-th body-fixed frame. 

The \textit{hat} map $\wedge :\Re^{3}\rightarrow\so$ transforms a vector in $\Re^3$ to a $3\times 3$ skew-symmetric matrix such that $\hat x y = (x)^\wedge y= x\times y$ for any $x,y\in\Re^3$. 
The inverse of the hat map is denoted by the \textit{vee} map $\vee:\so\rightarrow\Re^3$.
A few properties of the hat map are summarized as follows:
	\begin{gather}
    \widehat{x\times y} = \hat x \hat y -\hat y \hat x = yx^T-xy^T,\label{eq:hatxy}\\
    \tr[{\hat x A}]=
	\frac{1}{2}\tr[{\hat x (A-A^T)}]=-x^T (A-A^T)^\vee,\label{eq:hatxA}\\
	\hat{x}A+A^T\hat{x}=(\braces{\tr[A]I_{3\times3}-A}x)^\wedge, \label{eq:xAAx}\\
	R\hat x R^T = (Rx)^\wedge,\label{eq:RxR} \\
	\exp(\widehat{Ry})= R\exp(\hat{y})R^\T, \label{eq:exp}
	\end{gather}
for any $x,y\in\Re^3$, $A\in\Re^{3\times 3}$, and $R\in\SO$. 

Suppose that $n$ spacecraft are serially connected by daisy-chaining. For notational convenience, it is assumed that spacecraft indices are ordered along the serial network. The first spacecraft of the serial chain, namely Spacecraft 1 is selected to be the \textit{leader}, and its absolute position and its absolute attitude are controlled with respect to the inertial frame. The remaining spacecraft are considered as \textit{followers}, where each follower controls its relative position and its relative attitude with respect to the spacecraft that is one step ahead in the serial chain. For example, Spacecraft 2 controls its relative configuration with respect to Spacecraft 1.

It can be shown that the controller structure of each of followers are identical. To make the subsequent derivations more concrete and concise, the control systems are defined for the leader, Spacecraft 1, and the first follower, Spacecraft 2, only. Later, it is  generalized for control systems of the remaining spacecraft.

Each spacecraft is assumed to measure the line-of-sight toward two assigned objects via vision-based sensors to control its attitude. Line-of-sight measurements for the leader and the followers are described as follows.

\setlength{\unitlength}{0.1\columnwidth}
\begin{figure}
\scriptsize\selectfont
\centerline{
	\begin{picture}(10,5.8)(-0.2,0)
	\put(0.7,0){\includegraphics[width=0.75\columnwidth]{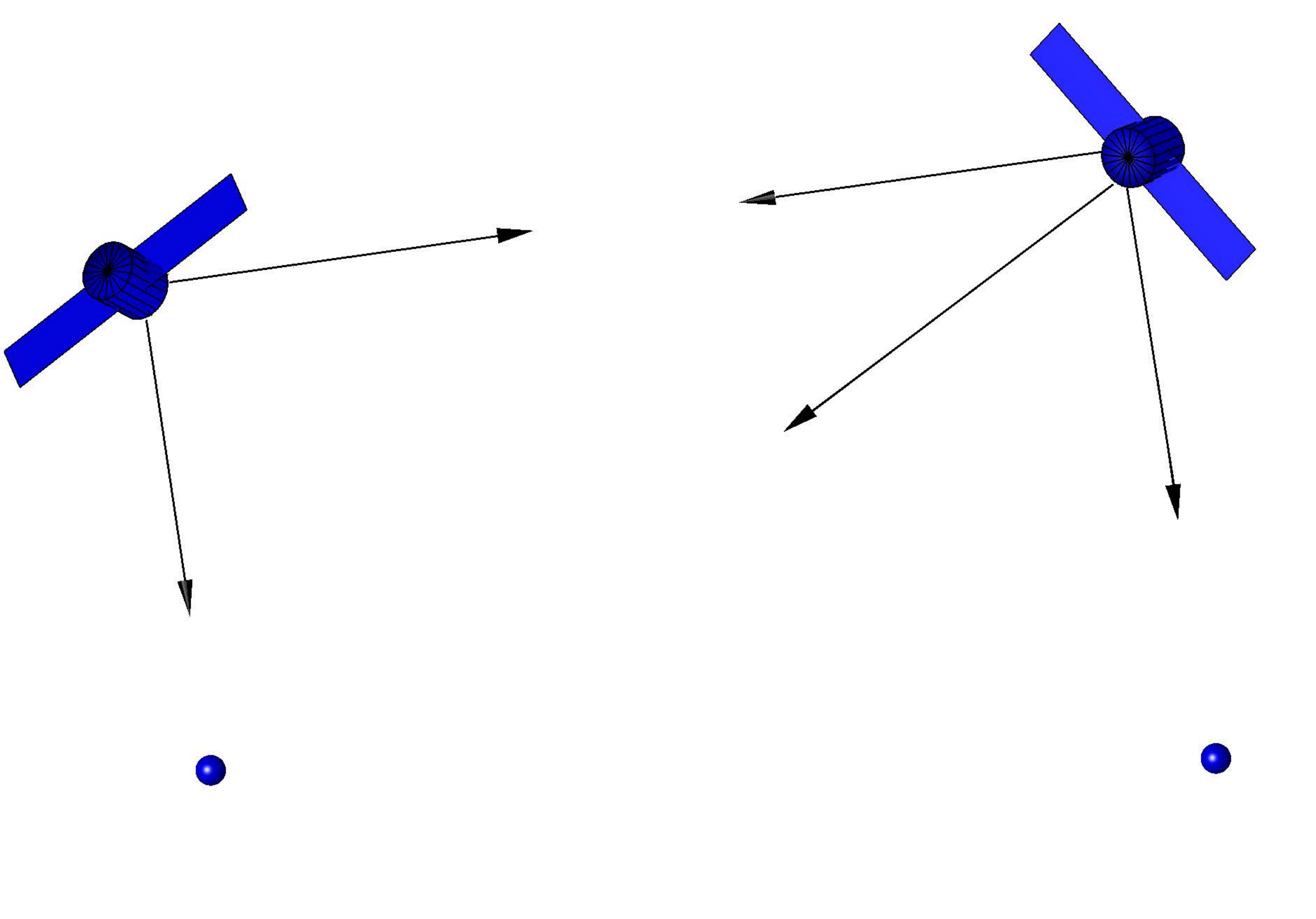}}
	\put(1.0,0.2){Object A (Spacecraft 3)}
	\put(7.1,0.3){Object B}
	\put(7.4,4.7){Spacecraft 1}
	\put(-0.2,4.0){Spacecraft 2}
	\put(1.9,2.3){$s_{23}$}
	\put(2.9,4.1){$s_{21}$}
	\put(5.2,4.4){$s_{12}$}
	\put(7.6,2.6){$s_{B}$}
	\put(4.5,2.5){$s_{A}\,(s_{13})$}
	\end{picture}}
\vspace*{-0.1cm}
\caption{Line-of-sight (LOS) Measurements: The leader, Spacecraft 1, measures the LOS toward to distinct objects $A,B$ to control its absolute attitude (the object $A$ is selected as Spacecraft 3). To control the relative attitude between Spacecraft 1 and Spacecraft 2, they measure the LOS toward each other, and also toward the common object selected as Spacecraft 3.} 
\label{fig:FS}
\end{figure}
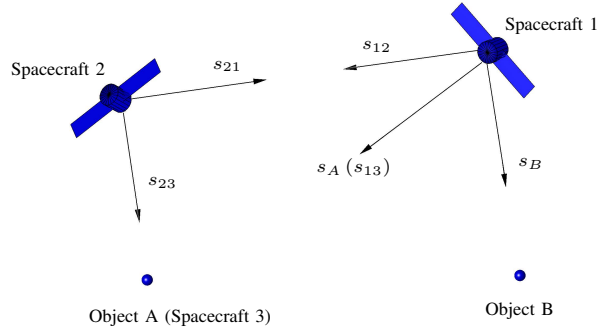

\subsection{Line-of-Sight Measurements of Leader}

Suppose that there are two distinct objects, namely $A,B$, such as distant stars, whose locations in the inertial reference frame are available. Let $s_A,s_B\in\Sph^2=\{q\in\Re^3\,|\,\|q\|=1\}$ be the unit-vectors representing the directions from Spacecraft 1 to the object A and to the object B, respectively. The vectors $s_A$ and $s_B$ are expressed with respect to the inertial frame, and we have $s_A\times s_B\neq 0$, as they are distinct objects. The time-derivatives of these unit-vectors are given by 
	\begin{align}
	\dot{s}_A=\mu_A\t s_A, \quad \dot{s}_B=\mu_B\t s_B, \label{eq:mu}
	\end{align}
where $\mu_A,\mu_B\in\Re^3$ are angular velocities of $s_A, s_B$, respectively, that can be determined explicitly by the relative translational motion of the leader with respect to the objects $A, B$. 

Assume the leader is equipped with a vision-based sensor that can measure the directions toward the objects $A$ and $B$. These two line-of-sight measurements are expressed with respect to the first body-fixed frame, and they are defined as $b_A, b_B\in\Sph^2$. Note that $s_A$, $s_B$ and $b_A$, $b_B$ are related by the rotation matrix, i.e., 
	\begin{gather} 
	s_j=R_1b_j \text{ or equivalently, } b_j=R_1^{\T}s_j, \label{eq:sRb}
	\end{gather}
for $j\in\{A,B\}$.  From (\ref{eq:mu}) and (\ref{eq:sRb}), we can obtain the kinematic equations for $b_A$, $b_B$ as
	\begin{align} 
	\dot{b}_j
	&= \dot{R}_1^{\T}s_j+R_1^{\T}\dot{s}_j =-\hat{\W}_1R_1^{\T}s_j  +R_1^\T\hat{\mu}_j s_j \no\\
	&= -\hat{\W}_1b_j +R_1^\T\hat{\mu}_j R_1R_1^\T s_j =b_j\t(\W_1-R_1^\T\mu_j), \label{eq:bi-dot}
	\end{align}
for $j\in\{A,B\}$. 

\subsection{Line-of-Sight Measurements of Follower}

Spacecraft 2 controls its relative attitude and its relative position with respect to Spacecraft 1. The relative configurations are defined as follows. The relative attitude of Spacecraft 2 with respect to Spacecraft 1 is represented by a rotation matrix $Q_{21}\in\SO$, given b
	\begin{align} 
	Q_{21}=R_1^\T R_2.
	\end{align}
From \refeqn{Rdot}, the time-derivative of $Q_{21}$ is given by
	\begin{align}
	\dot{Q}_{21} 
	&= -\hat\W_1 R_1^\T R_2 + R_1^\T R_2\hat\W_2 = Q_{21}\hat\W_2 - \hat\W_2 Q_{21}\no\\
	&= Q_{21} (\W_2 - Q_{21}^\T \W_1)^\wedge \triangleq Q_{21}\hat\W_{21},\label{eqn:dotQij}
	\end{align}
where the relative angular velocity vector of Spacecraft 2 with respect to Spacecraft 1 is defined as $\W_{21}=\W_2 -Q_{21}^\T\W_1\in\Re^3$. The relative position of Spacecraft 2 with respect to Spacecraft 1 is given by
\begin{align} 
	x_{21}=x_2-x_1.
	\end{align}


To control the relative attitude between the pair of Spacecraft 1 and Spacecraft 2, a third common object, namely Spacecraft 3 is assigned to form a triangular structure. It is assumed that Spacecraft 1 and Spacecraft 2 measure the line-of-sight toward each tother, and also toward Spacecraft 3. Furthermore, Spacecraft 3 does not lie on the line joining Spacecraft 1 and 2. 

Let $s_{ij}\in\Sph^2$ be the the unit-vector from the $i$-th spacecraft toward the $j$-th spacecraft, i.e., $s_{ij}=\frac{x_j-x_i}{\|x_j-x_i\|}$, and let $b_{ij}\in\Sph^2$ be the line-of-sight measurement from the $i$-th spacecraft toward the $j$-th spacecraft, represented with respect to the $i$-th body-fixed frame, i.e., $s_{ij}=R_ib_{ij}$. According to the above assumptions, we have $s_{13}\times s_{23}\neq 0$, and there are four measurements $\{b_{12},\, b_{13},\, b_{21},\, b_{23}\}$ for the pair of Spacecraft 1 and 2. 

It has been shown that the relative attitude $Q_{21}$ can be completely determined from the following constraints~\cite{LeePACC12}:
	\begin{gather}
	b_{12} = - Q_{21} b_{21}, \label{eq:consta}\\
	b_{123}= -Q_{21} b_{213},\label{eq:constb}
	\end{gather}	
where $b_{123}=\frac{b_{12}\t b_{13}}{\|b_{12}\t b_{13}\|}\in\Sph^2$	and $b_{213}=\frac{b_{21}\t b_{23}}{\|b_{21}\t b_{23}\|}\in\Sph^2$. The first constraint \refeqn{consta} states that the
unit vector from Spacecraft 1 to Spacecraft 2 is opposite to the unit vector from Spacecraft 2 to Spacecraft 1, i.e., $s_{12} = -s_{21}$. The second constraint \refeqn{constb} implies that the plane spanned by $s_{12}$ and  $s_{13}$ should be co-planar
with the plane spanned by  $s_{21}$ and  $s_{23}$.  For given LOS measurements $\{b_{12},b_{13},b_{21},b_{23}\}$, the relative attitude $Q_{21}$ is uniquely determined by solving \refeqn{consta} and \refeqn{constb} for $Q_{21}$.

Similar to the definitions in \refeqn{sRb}-\refeqn{bi-dot}, we have 
	\begin{gather} 
	s_{ij}(t)=R_i(t)b_{ij}(t),\quad b_{ij}(t)=R_i(t)^{\T}s_{ij}(t), \label{eq:rel_sRb} \\
	\dot{s}_{ij} =\mu_{ij}\t s_{ij}, \label{eq:rel_ds}\\
	\dot{b}_{ij} =b_{ij}\t(\W_i-R_i^\T\mu_{ij}), \label{eq:rel_db}
	\end{gather}
for $(i,j)\in\{(1,2), (1,3), (2,1), (2,3) \}$. Further, for $(i,j,k)\in\{(1,2,3), (2,1,3)\}$, we have  
	\begin{align}
	\dot{s}_{ijk}\triangleq\mu_{ijk}\t s_{ijk},
	\end{align}
where $s_{ijk}=\frac{s_{ij}\t s_{ik}}{\|s_{ij}\t s_{ik}\|}=R_ib_{ijk}\in\Sph^2$ and $\mu_{ijk}\in\Re^3$ is the angular velocity of $s_{ijk}$. Similar to \refeqn{rel_db}, we can write  
	\begin{align}
	\dot b_{ijk} &= \dot{R}_i^\T s_{ijk} +R_i^\T\dot{s}_{ijk} = b_{ij}\t(\W_i-R_i^\T\mu_{ijk}) 
	.\label{eqn:dotbijk}
	\end{align}
Notice that $\|b_{ij}\times b_{ik}\| = \| R_i^\T s_{ij} \times R_i^T s_{ik}\|=\|s_{ij}\times s_{ik}\|$, and it is non-zero. 


\subsection{Spacecraft Formation Control Problem}

Next, we define formation control problem. For the leader, Spacecraft 1, the desired \textit{absolute} attitude trajectory $R_1^d(t)\in\SO$ is given, and it satisfies the following kinematic equation
	\begin{gather} 
	\dot{R}_1^d=R_1^d\hat{\W}_1^d, \label{eq:Tra1} 
	\end{gather}
where $\W_1^d\in\Re^3$ is the desired angular velocity of Spacecraft 1. We transform this into the desired line-of-sight and its angular velocity as follows. The corresponding desired line-of-sight measurements are given by 
\begin{align} 
b_i^d=(R_1^d)^{\T}s_i,\quad i\in\{A,B\}. \label{eq:Tra2}
\end{align}
Similar with (\ref{eq:bi-dot}), the kinematic equations for desired line-of-sight measurements can be obtained as
\begin{align} 
	\dot{b}_i^d=b_i^d\t\big(\W_1^d-R_1^d\mu_i\big),\quad i\in\{A,B\}.
\end{align}
The desired position and the desired velocity of Spacecraft 1 are given by $x_1^d(t),v_1^d(t)\in\Re^3$.

For the follower, Spacecraft, the desired \textit{relative} attitude trajectory is defined as $Q_{21}^d(t)\in\SO$, and it satisfies the following kinematics equation,
\begin{align}
	\dot{Q}_{21}^d =Q_{21}^d\hat{\W}_{21}^d,
\end{align}
where $\hat{\W}_{21}^d\in\Re^3$ is the desired relative angular velocity vector, satisfying
	\begin{align}
	\W_{21}^d = \W_2^d - {Q_{21}^d}^\T\W_1.\label{eq:Wijd}
	\end{align}	
Let the desired relative position and the relative velocity of Spacecraft 2 be $x_{21}^d(t),v_{21}^d(t)\in\Re^3$, and they satisfy
\begin{align}
	\dot{x}_{21}^d &=v_{21}^d =v_2^d-v_1.
	\end{align}	

\begin{assump}
The magnitude of each desired angular velocity is bounded by
	\begin{align}
	\|\W^d\|\leq \Bd,  \label{eq:Bd}
	\end{align}
where $\W^d\in\{\W_{21}^d, \W_2^d, \W_1^d\}$.
\end{assump}
\begin{assump}
The magnitude of angular velocity of each line-of-sight measurement is bounded by 
	\begin{align}
	\|\mu\|\leq B_\mu,  \label{eqn:mu}
	\end{align}
where $B_\mu\in\Re^+$ is a positive constant. The $\mu$ here includes $\{\mu_A, \mu_B, \mu_{12}, \mu_{21},\mu_{123}, \mu_{213}\}$.	
\end{assump}
The first assumption is natural in any tracking problem, and the second assumption is satisfied if there is no collision between spacecraft and the velocities of each spacecraft are bounded.

The goal is to design control inputs ($u_1,\,u_2,\,f_1,\,f_2$) such that the zero equilibrium of the position and attitude tracking errors becomes asymptotically stable, and in particular, it is required the the control moments $u_1,u_2$ are directly expressed in terms of the line-of-sight measurements.


\section{Attitude and Position Tracking Error Variables}\label{sec:AbAtt}

As a preliminary work before designing control systems, in this section, we present error variables for each of absolute attitude tracking error, relative attitude tracking error, and position tracking error. 

\subsection{\tcb{Absolute Attitude Error Variables}}

The fundamental idea of constructing attitude control system in terms of LOS measurements is utilizing the property that the attitude error becomes zero if the LOS are aligned to their desired values, i.e., we have $R_1=R_1^d$ if $b_A=b_A^d$ and $b_B=b_B^d$~\cite{Wu2012}. This motivates the following definition of the configuration error function for each line-of-sight:
		\begin{gather} 
		\Psi_j=1-b_j\cdot b_j^d =1-{R_1^\T}s_j\cdot{R_1^d{^\T}}s_j,
		\end{gather}
for each object $j\in\{A,B\}$.	They are combined into	
\begin{gather}
\Psi_1=k_{b_A}\Psi_A +k_{b_B}\Psi_B, \label{eq:error2}
\end{gather}
where $k_{b_A}\neq k_{b_B}>$ are positive constants. The error function $\Psi_i$ refers to the corresponding error of LOS measurements, and $\Psi_1$ represents the combined errors for Spacecraft 1.

By finding the derivatives of the configuration error functions, we obtain the configuration error vectors as folllows.
		\begin{gather} 
		e_{b_A}=b_A\t b_A^d,\quad e_{b_B}=b_B\t b_B^d, \label{eq:error3}\\
		e_b= k_{b_A}e_{b_A}+k_{b_B}e_{b_B}. \label{eq:error4}
		\end{gather}
Additionally, the angular velocity error vector is defined as 
	\begin{align}
	e_{\W_1}=\W_1-\W_1^d. \label{eq:error5}
	\end{align}

It can be shown that the above error variables satisfy the following properties.

\begin{prop}\label{prop:Att_Err} 

\renewcommand{\labelenumi}{(\roman{enumi})}
\begin{enumerate}\renewcommand{\itemsep}{3pt}
\item The error function  $\Psi_1$ and the error vector $e_b$ can be rewritten as
	\begin{gather}
	\Psi_1=\tr[K_1(\I-R_1R_1^d{}^\T)], \label{eq:Psi}\\
 	e_b=[R_1^d{^\T}K_1R_1-R_1^{\T}K_1R_1^d]^\vee, \label{eq:eb}		
	\end{gather}
where $K_1=k_{b_A}s_As_A^{\T}+k_{b_B}s_Bs_B^{\T}\in\Re^{3\t3}$ is symmetric and $\I$ is the 3 by 3 identity matrix.
\item $\Psi_1$ is locally quadratic, more expliclty,
	\begin{gather}
	\underline\psi_{b}\|e_{b}\|^2\leq \Psi_{1} \leq \overline\psi_{b} \|e_{b}\|^2, \label{eq:psi1}
	\end{gather}		
where $\underline\psi_{b}=\frac{h_1}{h_2+h_3}$ and $\overline\psi_{b}=\frac{h_1h_4}{h_5(h_1-\psi_1)}$ for		
	\begin{align*} 
	h_1&= 2\min\{k_{b_A},k_{b_B}\}, \\
	h_2&= 4\max\{(k_{b_A}-k_{b_B})^2, k_{b_A}^2, k_{b_B}^2\}, \\
	h_3&= 4(k_{b_A}+k_{b_B})^2, \\
	h_4&= 2(k_{b_A}+k_{b_B}), \\
	h_5&= 4\min\{k_{b_A}^2,k_{b_B}\},
	\end{align*}	
and $\psi_1$ is a positive constant satisfying $\Psi_1<\psi_1<h_1$. 
\item $\|e_b\| \leq k_{b_A}+k_{b_B}$.
\item $\frac{d}{dt}\Psi_1=e_{b}\cdot e_{\W_1} +\Gamma\|e_b\|$.
\item $\|\frac{d}{dt}e_b\|\leq\frac{1}{\sqrt{2}}(k_{b_A}+k_{b_B})\|e_{\W_1}\|+(B_{\W_d}+B_1)\|e_b\|$.	 
\end{enumerate}
\end{prop}   

	\begin{proof}
	See Appedix, Section \ref{subsec:pAtt} .
	\end{proof}

\subsection{\tcb{Relative Attitude Error Variables}}

Similarly, we define error variables for the relative attitude. It is also based on the fact that we have $Q_{21}=Q_{21}^d$ provided that $b_12=-Q_{21}^d b_{21}$ and $b_{123}=-Q_{21}^db_{213}$~\cite{WuACC13}. Configuration error functions that represent the errors in satisfaction of \refeqn{consta} and \refeqn{constb} are defined as
		\begin{gather}
		\Psi_{21}^\al = \frac{1}{2}\|b_{12}+Q_{21}^d b_{21}\|^2 = 1 +  b_{12}\cdot Q_{21}^d b_{21},\no\\
		\Psi_{21}^\be = 1 +b_{123}\cdot Q_{21}^d b_{213}.
		\end{gather}
For positive constants $k_{21}^\alpha\neq k_{21}^\beta$, these are combined into		
		\begin{gather}
		\Psi_{21} =k_{21}^\al\Psi_{21}^\al +k_{21}^\be\Psi_{21}^\be. \label{eq:Psi21}
		\end{gather}
These yield the configuration error vectors  as
		\begin{gather}
		e_{21}^\al =({Q_{21}^d}^\T b_{12})\t b_{21},\quad
		e_{21}^\be =({Q_{21}^d}^\T b_{123})\t b_{213}, \no\\
		e_{21} =k_{21}^\al e_{21}^\al +k_{21}^\be e_{21}^\be. \label{eq:e21}
		\end{gather}
	We have $\|e^\alpha_{21}\|,\|e^\beta_{21}\|\leq 1$ as $b_{ij},b_{ijk}$ are unit vectors. And the angular velocity error vector is defined as:
	\begin{align}
	e_{\Omega_2} = \Omega_2 -\Omega_2^d,\label{eq:eW12}
	\end{align}
where $\Omega_2^d$ can be determined by \refeqn{Wijd}. The properties of  relative error variables are summarized as follows:
	
\begin{prop}\label{prop:Rel_Err}	
\renewcommand{\labelenumi}{(\roman{enumi})}
\begin{enumerate}\renewcommand{\itemsep}{3pt}	
\item The error function $\Psi_{21}$ and the error vector $e_{21}$ can be rewritten as
		\begin{align*}
		\Psi_{21} &= \tr[K_{21}(\I-R_1Q_{21}^dR_2^\T)], \\
		e_{21} &=[{Q_{21}^d}^\T{R_1^\T}K_{21}R_2 -R_2^\T K_{21}{R_1}{Q_{21}^d}]^\vee,
		\end{align*}
	  where $K_{21}=k_{21}^\al s_{21}s_{21}^\T +k_{21}^\be s_{213}s_{213}^\T\in\Re^{3\t3}$.
\item $\Psi_{21}$ is locally quadratic, more expliclty,
	\begin{align} 
	\underline{\psi}_{21}\|e_{21}\|^2 \leq\Psi_{21}\leq \overline{\psi}_{21}\|e_{21}\|^2, \label{eq:psi21}
	\end{align}
where $\underline{\psi}_{21}=\frac{n_1}{n_2+n_3}$ and $\overline{\psi}_{21}=\frac{n_1n_4}{n_5(n_1-\phi_{21})}$, for
	\begin{align*} 
	n_1&= 2\min\{k_{21}^\al, k_{21}^\be\}, \\
	n_2&= 4\max\{(k_{21}^\al- k_{21}^\be)^2, (k_{21}^\al)^2, (k_{21}^\be)^2\}, \\
	n_3&= 4(k_{21}^\al+k_{21}^\be)^2, \\
	n_4&= 2(k_{21}^\al+k_{21}^\be), \\
	n_5&= 4\min\{{k_{21}^\al}^2, k_{21}^\be\},
	\end{align*}
	 and $\phi_{21}$ is a positive constant satisfying $\Psi_{21}<\phi_{21}<n_1$. 	  	
\item $\frac{d}{dt}\Psi_{21}=e_{21}\cdot e_{\W_2} +\Gamma_{21}\|e_{21}\|^2$.
\item $\|\frac{d}{dt}e_{21}\|=\frac{1}{\sqrt{2}}(k_{21}^\al+k_{21}^\be)\|e_{\W_2}\| +(\Bd+B_{21})\|e_{21}\|$.
\end{enumerate}	
\end{prop}
	\begin{proof}
	Seen Appendix, Section \ref{subsec:pRel} .
	\end{proof}

\subsection{\tcb{Position Tracking Error Variables}}
The position and velocity error vectors of Spacecraft 1 are given as 
	\begin{align}
	e_{x_1} &=x_1-x_1^d, \\
	e_{v_1} &=v_1-v_1^d =\dot{e}_{x_1}.\label{eq:dex1} 
	\end{align} 
Similarly, the relative position and relative velocity error vectors are given by
	\begin{align}
	e_{x_{21}} &=x_{21}-x_{21}^d, \\
	e_{v_{21}} &=v_{21}-v_{21}^d =\dot{e}_{x_{21}}, \label{eq:dex2} 
	\end{align} 

\section{Formation Control System}

Based on the error variables defined at the previous section, control systems are designed as follows. 

\subsection{Formation Control for Two Spacecraft}
\begin{prop} \label{prop:controller}
Consider the formation of two spacecraft shown in Fig \ref{fig:FS}. For positive constants, $k_{\W_i}$, $k_{x_i}$, and $k_{v_i}$, $i\in\{1,2\}$, the control inputs are chosen as follows
	\begin{align}
 	u_1 &=-e_b -k_{\W_1} e_{\W_1} +\hat\W^d_1J_1(e_{\W_1}+\W^d_1) +J_1\dot\W^d_1,  \label{eq:u1}\\
 	u_2 &=-e_{21} -k_{\W_2} e_{\W_2} +\hat\W^d_2J_2(e_{\W_2}+\W^d_2) +J_2\dot\W^d_2,  \label{eq:u2} \\
 	f_1 &=-k_{x_1} e_{x_{1}} -k_{v_1}e_{v_1} +m_1\ddot{x}_1^d, \label{eq:f1} \\
 	f_2 &=-k_{x_2} e_{x_{21}} -k_{v_2}e_{v_{21}} +m_2(\ddot{x}_1+\ddot{x}_{21}^d). \label{eq:f2}
 	\end{align}
Then, the zero equilibrium of tracking errors is exponentially stable. 	
\end{prop}

\begin{proof}
The error dynamics are as follows. From \refeqn{f}, we have
	\begin{align}
	\dot{e}_{v_1}=\frac{f_1}{m_1}-\ddot{x}_{1}^d, \quad 
	\dot{e}_{v_{21}}=\frac{f_2}{m_2}-\ddot{x}_{1}-\ddot{x}_{21}^d, \label{eq:useful1}
	\end{align}
By using \refeqn{Wdot}, \refeqn{error5}, \refeqn{eW12}, we can obtain
	\begin{align}
	J_i\dot{e}_{\W_i} 
	&=[J(e_{\W_i}+\W_i^d)]^\wedge e_{\W_i} -\hat{\W}_i^dJ_i(e_{\W_i} +\W_i^d) \no\\
	&\quad -J_i\dot{\W}_i^d +u_i, \label{eq:useful2}
	\end{align}
for $i=\{1,2\}$. 
	
For positive constants $c_r$ and $c_t$, let the Lyapunov function be
\begin{align}
\V=\V_{1}^r+\V_{21}^r +\V^t_{1}+\V^t_{21},
\end{align}
where
\begin{align}
	\V_{1}^r &= \frac{1}{2}e_{\W_1}\cdot J_1e_{\W_1} +\Psi_{1} +c_rJ_1e_{\W_1}\cdot e_{b},  \label{eq:v1r}\\
	\V_{21}^r &= \frac{1}{2}e_{\W_2}\cdot J_2e_{\W_2} +\Psi_{21} +c_rJ_2e_{\W_2}\cdot e_{21},  \label{eq:v21r}\\
	\V^t_{1} &=\frac{1}{2}k_{x_1}e_{x_{1}}^\T e_{x_{1}} +\frac{1}{2}m_1e_{v_{1}}^\T e_{v_{1}}, +c_te_{x_{1}}^\T e_{v_{1}}, \label{eq:v1t}\\
	\V^t_{21} &=\frac{1}{2}k_{x_2}e_{x_{21}}^\T e_{x_{21}} +\frac{1}{2}m_2e_{v_{21}}^\T e_{v_{21}} +c_te_{x_{21}}^\T e_{v_{21}}. \label{eq:v21t}
	\end{align}
From \refeqn{psi1}, \refeqn{psi21}, the Lyapunov function is positive-definite about the equilibrium $(e_b,\,e_{21},\,e_{\W_i},\,e_{x_i},\,e_{v_i})=(0,0,0,0,0)$ for $i\in\{1,2\}$, provided that the constants $c_r,c_t$ are sufficiently small. 
		
The time-derivative of $\V_1^r$, $\V_{21}^r$ are given by	
	\begin{align*} 
	\dot{\V}_{1}^r 
	&=(e_{\W_1}+c_re_b)^\T J_1\dot{e}_{\W_1} +\dot\Psi_1 +c_rJ_1e_{\W_1}^\T\dot{e}_b, \\
	\dot{\V}_{21}^r 
	&=(e_{\W_2}+c_re_{21})^\T J_2\dot{e}_{\W_2} +\dot\Psi_{21} +c_rJ_2e_{\W_2}^\T\dot{e}_{21}.
	\end{align*}
By using the properties (iv),(v) of Proposition \ref{prop:Att_Err}, and the properties (iii),(iv) of Proposition \ref{prop:Rel_Err}, and also substituting \refeqn{useful2}, \refeqn{u1}, and \refeqn{u2}, the time-derivative of $\V_1^r$, $\V_{21}^r$ can be rearranged as
	\begin{align}
		\dot{\V}_{1}^r 
		&\leq -[k_{\W_1} -c_r\lam_{M_1}(\frac{1}{\sqrt{2}}+1)\bar{k}_{b} ]\|e_{\W_1}\|^2 \no\\
		&\quad +c_r[\lam_{M_1}(2\Bd +B_{1})+k_{\W_1}]\|e_{\W_1}\|\|e_{b}\| \no\\
		&\quad -(c_r-\Gamma)\|e_{b}\|^2, \label{eq:dv1r} \\
	\dot{\V}_{21}^r 
	&\leq -[k_{\W_2} -c_r\lam_{M_2}(\frac{1}{\sqrt{2}}+1)\bar{k}_{21} ]\|e_{\W_2}\|^2 \no\\
	&\quad +c_r[\lam_{M_2}(2\Bd +B_{21})+k_{\W_2}]\|e_{\W_2}\|\|e_{21}\| \no\\
	&\quad -(c_r-\Gamma_{21})\|e_{21}\|^2, \label{eq:dv2r}
	\end{align}	
	
For the translational dynamics, from \refeqn{dex1} and \refeqn{dex2}, we have	
	\begin{align*} 	
	\dot{\V}_{1}^t 
	&= k_{x_1}e_{x_1}^\T e_{v_1} +c_t\|e_{v_1}\|^2 +(m_1e_{v_1}+ce_{x_1})^\T\dot{e}_{v_1} \label{}\\
	\dot{\V}_{21}^t 
	&= k_{x_2}e_{x_{21}}^\T e_{v_{21}} +c_t\|e_{v_{21}}\|^2 +(m_2e_{v_{21}}+ce_{x_{21}})^\T\dot{e}_{v_{21}}.
	\end{align*}
Substituting \refeqn{useful1} with the control inputs \refeqn{f1}, \refeqn{f2}, we obtain
	\begin{align} 
	\dot{\V}_{1}^t 
	&=-\frac{ck_{x_1}}{m_1}\|e_{x_1}\|^2 -(k_{v_1}-c_t)\|e_{v_1}\|^2 -\frac{c_tk_{v_1}}{m_1}e_{x_1}^\T e_{v_1},\label{eq:dv1t}\\
	\dot{\V}_{21}^t 
	&=-\frac{ck_{x_2}}{m_2}\|e_{x_{21}}\|^2 -(k_{v_2}-c_t)\|e_{v_{21}}\|^2 -\frac{c_tk_{v_2}}{m_2}e_{x_{21}}^\T e_{v_{21}},\label{eq:dv2t}
	\end{align}
	
From \refeqn{dv1r}, \refeqn{dv2r}, \refeqn{dv1t} and \refeqn{dv2t}, the time-derivative of the complete Lyapunov function $\V$ can be written as
	\begin{align} 
	\dot{\V}\leq-(\zeta_1^\T{M_1}\zeta_1 +\zeta_{21}^\T{M_{21}}\zeta_{21} +\xi_1^\T N_1\xi_1 +\xi_{21}^\T N_{21}\xi_{21}),
	\end{align}	
where $\zeta_1=[\|e_b\|, \|e_{\W_1}]^\T$, $\zeta_{21}=[\|e_{21}\|, \|e_{\W_2}\|]^\T$, $\xi_1=[\|e_{x_1}\|,\|e_{v_1}\|]^\T$ and $\xi_{21}=[\|e_{x_{21}}\|,\|e_{v_{21}}\|]^\T\in\Re^2$, and the matrices $M_1,M_{21},N_1,N_{21}\in\Re^{2\times 2}$ are defined as
\begin{align*} 
    M_1 &=\frac{1}{2}\left[\begin{array}{cc}2(c_r-\Gamma) & -c_r\Lambda_1 \\ -c_r\Lambda_1 & 2k_{\W_1}-c_r\lam_{M_1}\bar{k}_{b}(\sqrt{2}+2) \end{array}\right], \\
    M_{21} &=\frac{1}{2}\left[\begin{array}{cc}2(c_r-\Gamma_{21}) & -c_r\Lambda_2 \\ -c_r\Lambda_2 & 2k_{\W_2}-c_r\lam_{M_2}\bar{k}_{21}(\sqrt{2}+2) \end{array}\right], \\
    N_1 &=\frac{1}{2m_1}\left[\begin{array}{cc} 2c_tk_{x_1} & c_tk_{v_1} \\ c_tk_{v_1} & 2m_1(k_{v_1}-c_t)\end{array}\right], \\
    N_{21} &=\frac{1}{2m_2}\left[\begin{array}{cc} 2c_tk_{x_{21}} & c_tk_{v_{21}} \\ c_tk_{v_{21}} & 2m_2(k_{v_2}-c_t)\end{array}\right], 
	\end{align*}	
where $\Lambda_1=\lam_{M_1}(2\Bd +B_{1})+k_{\W_1}\in\Re$ and $\Lambda_2=\lam_{M_2}(2\Bd +B_{21})+k_{\W_2}\in\Re$. If the constants $c_t$ and $c_r$ are sufficiently small, we can show that all of matrices $M_1$, $M_{21}$, $N_1$ and $N_{21}$ are positive definite. This implies that the equilibrium $(e_b,\,e_{21},\,e_{\W_i},\,e_{x_i},\,e_{v_i})=(0,0,0,0,0)$ for $i\in\{1,2\}$, i.e., the desired formation, is exponentially stable.

\end{proof}

\subsection{Formation Control for Multiple Spacecraft}

The preceding results for two spacecraft are readily generalized for an arbitrary number of spacecraft. Here, we present the controller structures as follows, without stability proof that can be obtained by generalizing the proof of Proposition \ref{prop:controller}.

\begin{prop} \label{prop:controller2}
Consider $n$ spacecraft in the formation, the $i$-th spacecraft is paired serially with the $(i-1)$-th spacecraft for $i=\in\{2,3,\cdots,n\}$. For positive constants, $k_{{\W}_1}$, $k_{x_1}$, $k_{v_i}$,  $k_{\W_i}$, $k_{x_i}$, and $k_{v_i}$,  for $i\in\{2,3,\cdots,n\}$, the control inputs are chosen as follows
	\begin{align}
 	u_1 &=-e_b -k_{\W_1} e_{\W_1} +\hat\W^d_1J_1(e_{\W_1}+\W^d_1) +J_1\dot\W^d_1, \\ 
 	u_i &=-e_{i,i-1} -k_{\W_i} e_{\W_i} +\hat\W^d_iJ_i(e_{\W_i}+\W^d_i) +J_i\dot\W^d_i, \\ 
 	f_1 &=-k_{x_1} e_{x_{1}} -k_{v_1}e_{v_1} +m_1\ddot{x}_1^d, \\ 
 	f_i &=-k_{x_i} e_{x_{i,i-1}} -k_{v_i}e_{v_{i,i-1}} +m_i(\ddot{x}_{i-1}+\ddot{x}_{i,i-1}^d). 
 	\end{align}
Then the zero equilibrium of tracking errors is exponentially stable.
\end{prop}

\section{Numerical Simulation}	

Two simulation results are presented: (i) formation tracking control for two spacecraft, and (ii) formation stabilization for four spacecraft. 

\subsection{Formation Tracking for Two Spacecraft}

Suppose $n=2$. The mass and the inertia matrix are chosen as $m_1=m_2=30\,\mathrm{kg}$ and $J_1=J_2=\mathrm{diag}[3,2,1]\,\mathrm{kgm^2}$. The desired absolute attitude of the leader, namely $R_1^d(t)$ is specified in terms of $3-2-1$ Euler angles $(\al(t),\,\be(t),\,\gamma(t))$, where 
\begin{gather*}
	\alpha(t)=0,\,\beta(t)=-0.7+\cos(0.2t),\,\gamma(t)=0.5+\sin(2t).
\end{gather*}
The desired relative attitude $Q_{21}^d(t)$ is also defined in terms of another set of Euler angles given by
	\begin{gather*}
	\phi(t)=\sin(0.5t),\,\theta(t)=2,\,\psi(t)=\cos(t)+1.
	\end{gather*}
The initial attitudes for Spacecraft 1 and Spacecraft 2 are chosen as $R_1(0)=R_2(0)=\I$. The initial angular velocity is chosen to be zero for both spacecraft.  

For the translational motion, the desired  position vectors are given by  
		\begin{gather*}
		x_1^d=[\sin(0.04t),\,0,\,-\sin(0.07t)]^\T\\
		x_{21}^d=[2,\, -3+\cos(0.02t),\, 10]^\T.
		\end{gather*}
The initial positions are chosen as $x_1=[0,\,0,\,0]$ and $x_2=[2,\,-1,\,7]$. Control gains are selected to be $k_{\W_1}=k_{\W_2}=7$, $k^\al_{21}=k_{b_1}=25$, $k^\beta_{21}=k_{b_2}=25.1$, $k_{x_1}=k_{x_2}=49$, and $k_{v_1}=k_{v_2}=12.6$. 

The corresponding numerical results are illustrated in Fig \ref{fig:err}, where the attitude error vectors are defined as 
		\begin{gather*} 
		e_{R_1}=\frac{1}{2}	({R_1^d}^\T R_1-R_1^\T R_1^d)^\vee, \\
		e_{Q_{21}}=\frac{1}{2}	({Q_{21}^d}^\T Q_{21}-Q_{21}^\T Q_{21}^d)^\vee.
		\end{gather*}
It is illustrated that tracking errors are nicely converted to zero. 
	
\begin{figure}
\centerline{
	\subfigure[Attitude error vector $e_R,e_{Q_{21}}$] {\includegraphics[width=0.49\columnwidth]{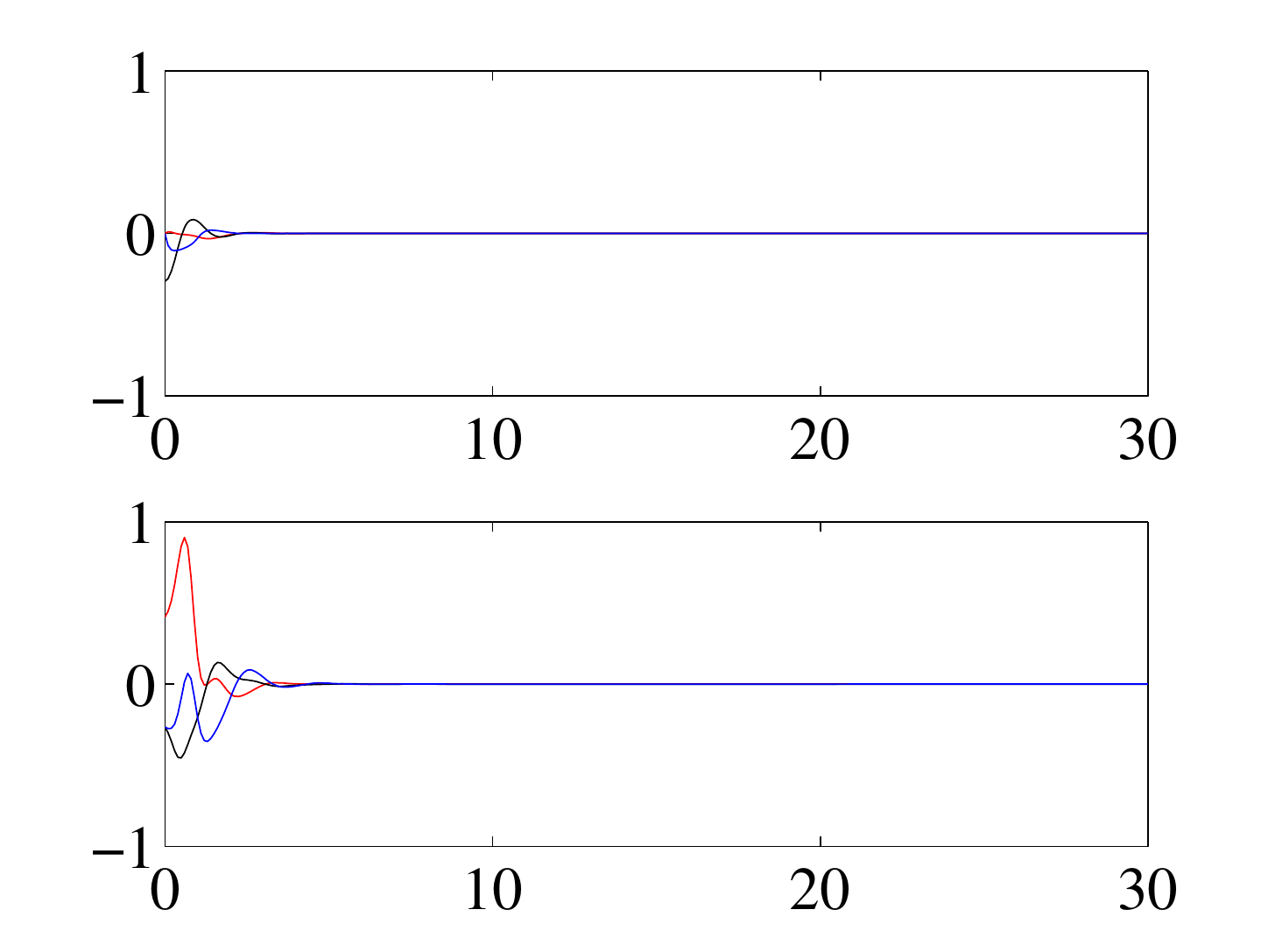}}
	\hspace*{0.001\columnwidth}
	\subfigure[Angular velocity error $e_{\W_1},e_{\W_2}$] {\includegraphics[width=0.49\columnwidth]{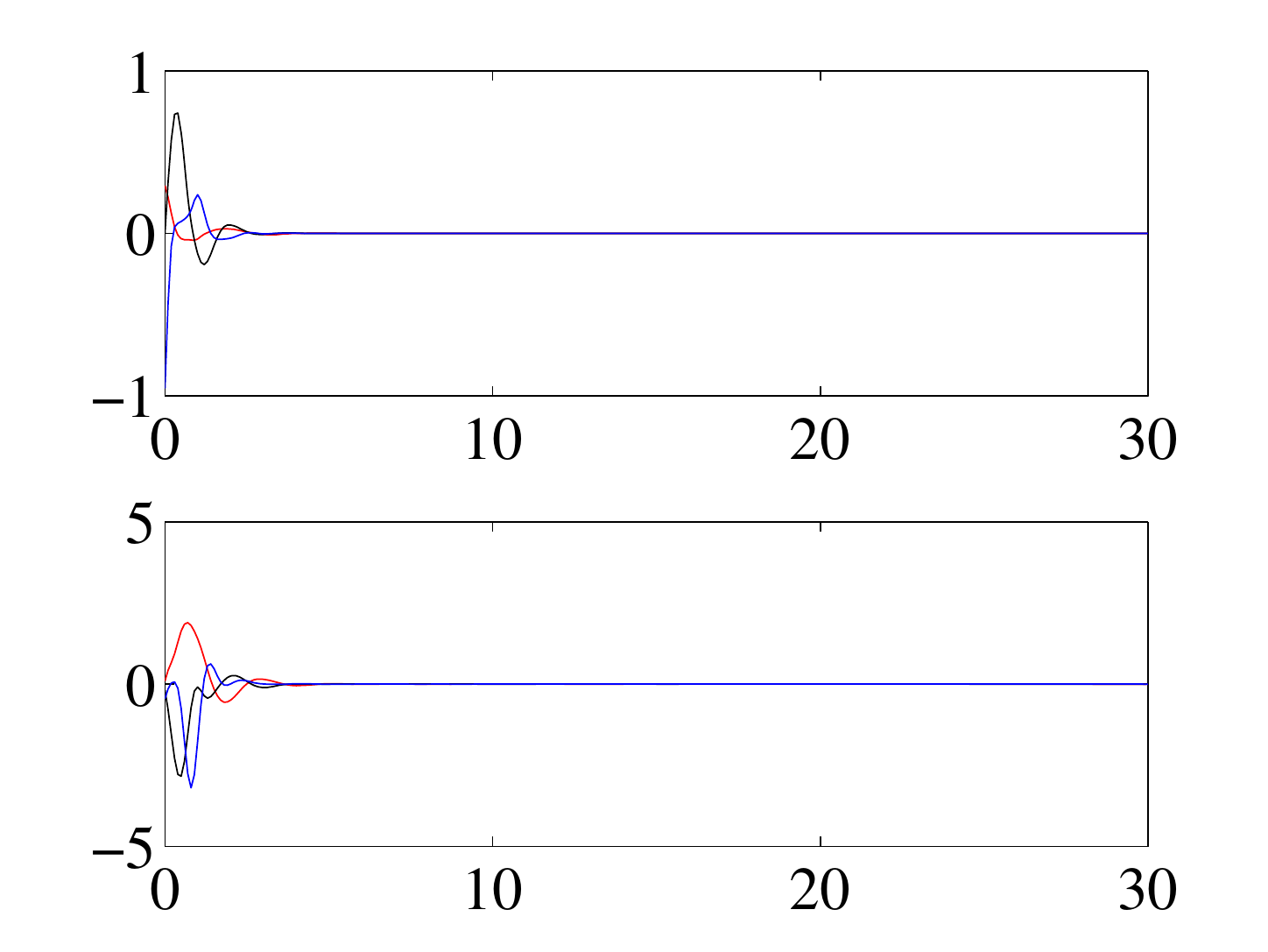}}
}
\centerline{
	\subfigure[Position error vector $e_{x_1},e_{x_{21}}$] {\includegraphics[width=0.47\columnwidth]{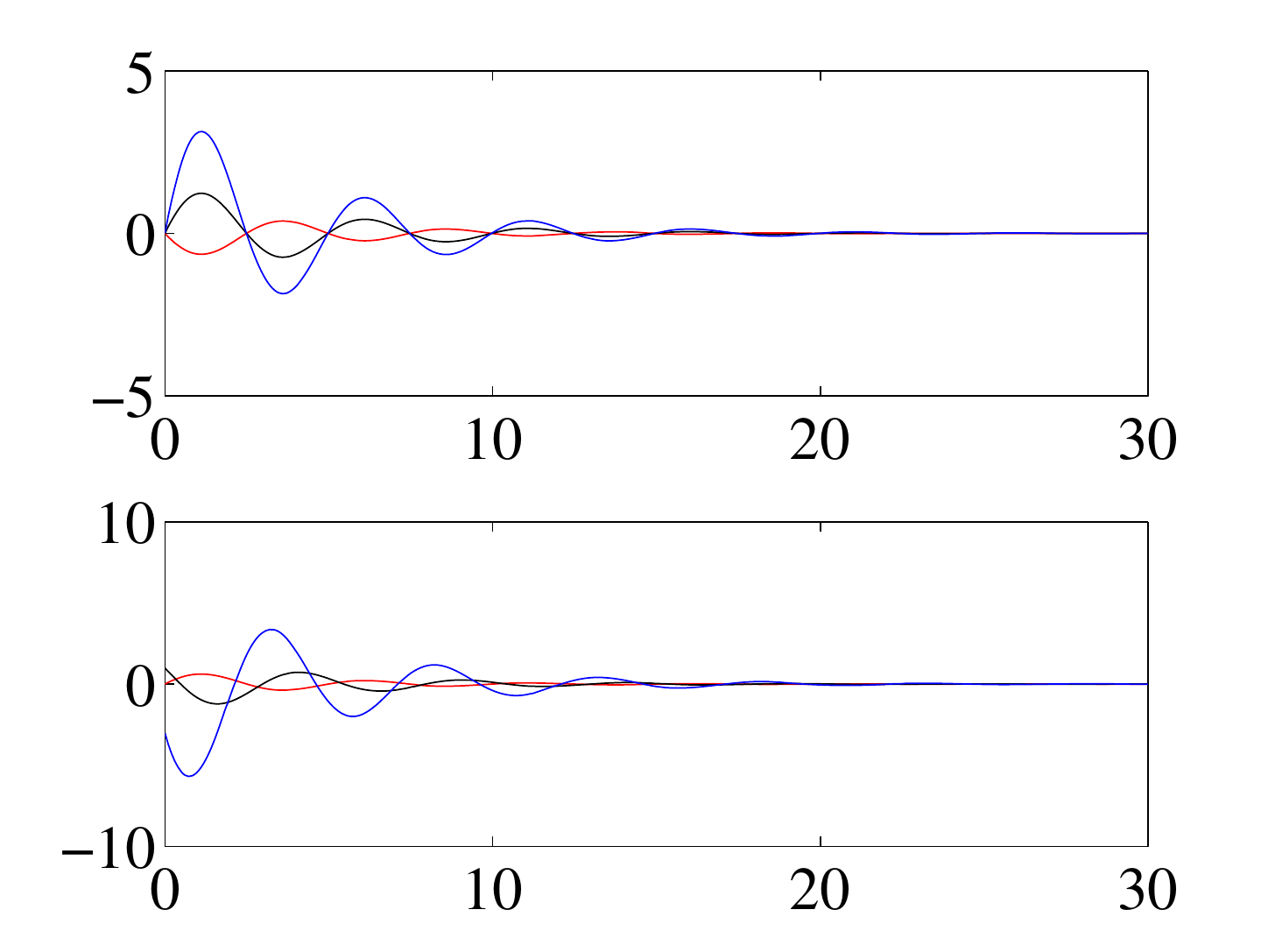}}
	\hspace*{0.03\columnwidth}
	\subfigure[Velocity error vector $e_{v_1},e_{v_{21}}$] {\includegraphics[width=0.47\columnwidth]{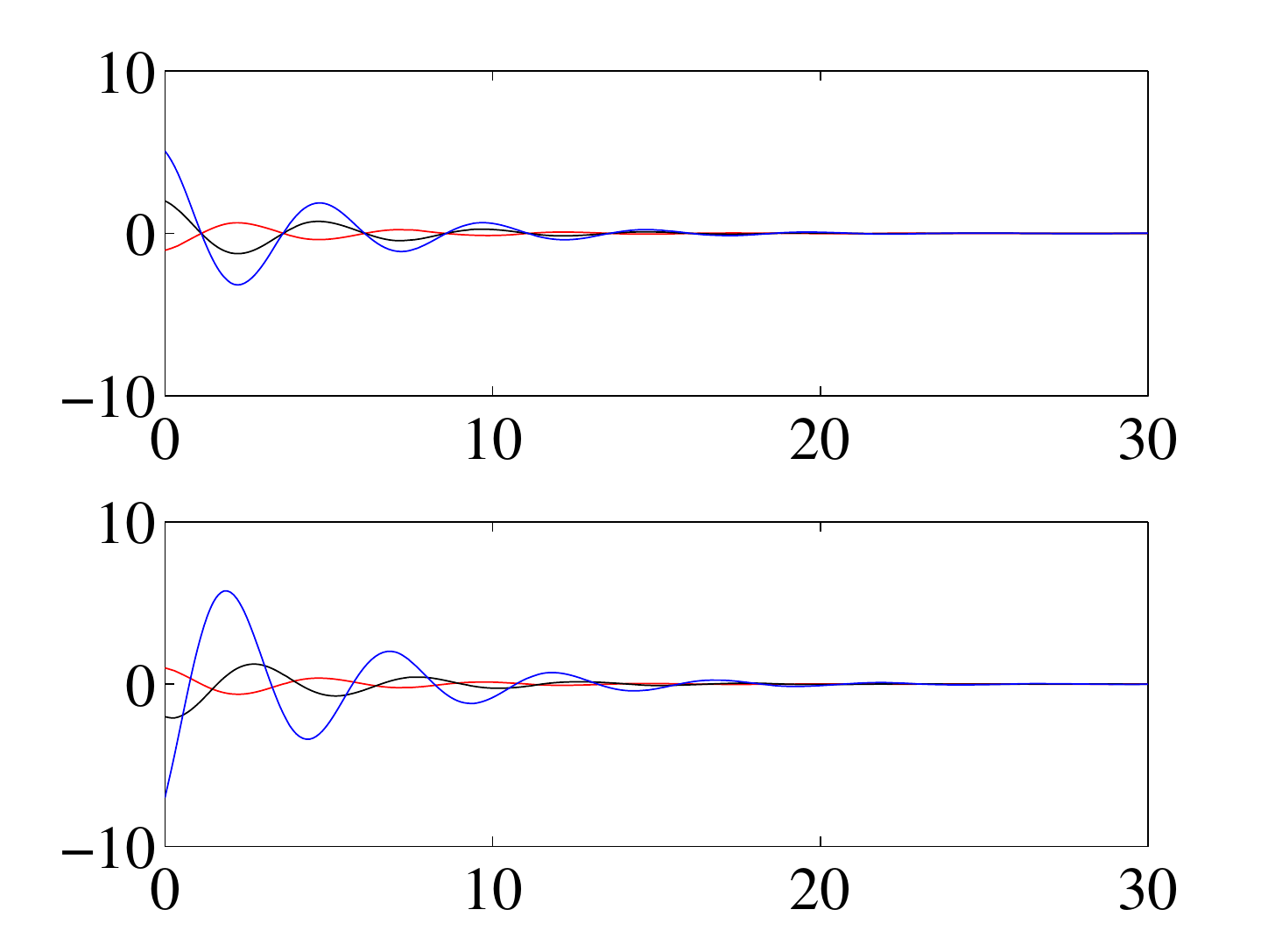}}
}
\centerline{
	\subfigure[Control moment $u_1,u_2$ (Nm)] {\includegraphics[width=0.47\columnwidth]{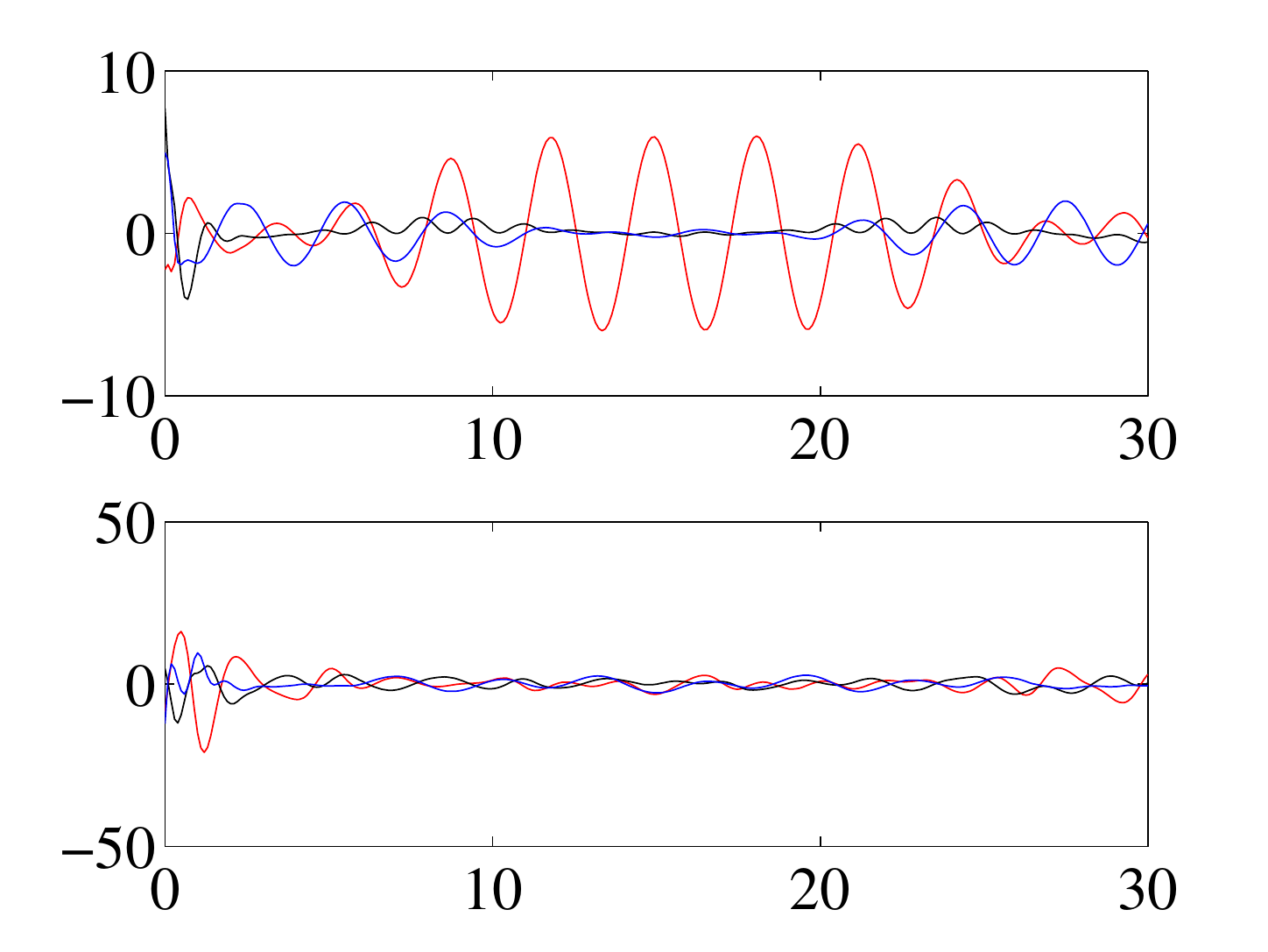}}
	\hspace*{0.03\columnwidth}
	\subfigure[Control force $f_1,f_2$ (N)] {\includegraphics[width=0.47\columnwidth]{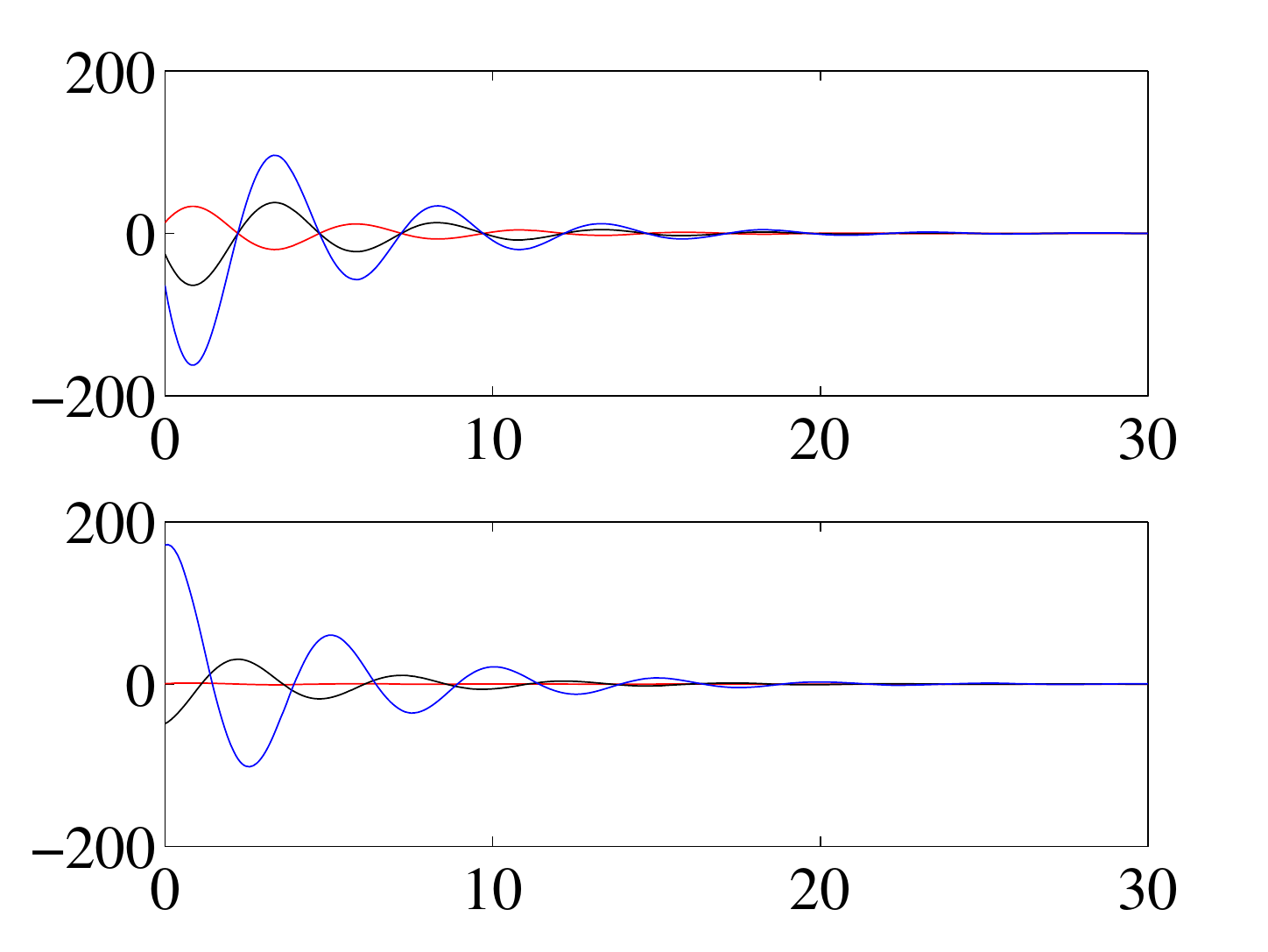}}
}
\caption{Numerical results for two spacecraft in formation.}\label{fig:err}
\vspace*{-0.1cm} 
\end{figure}

\subsection{Formation Control for Four Spacecraft} \label{subsec:quad}
Next, we consider formation control for $n=4$ spacecraft. The mass and the inertia matrix are same as the previous case. The desired attitudes are chosen as $R_1^d=Q_{12}^d=Q_{23}^d=Q_{34}^d=\I$. This represents attitude synchronization. The initial attitudes are given as 
	\begin{gather*} 
	R_1(0)=\exp(0.2\pi\hat{e}_2),\quad R_2(0)=\exp(0.5\pi\hat{e}_1), \\
	R_3(0)=\exp(0.4\pi\hat{e}_1),\quad R_3(0)=\exp(0.8\pi\hat{e}_3),
	\end{gather*}
where $e_1=[1,0,0]^\T$, $e_2=[0,1,0]^\T$, and $e_3=[0,0,1]^\T$. Also, the initial angular velocity is chosen to be zero for each spacecraft. 

The desired position trajectories are chosen as
	\begin{gather*} 
	x_1^d=[-100,\, 0,\,  0]^\T,\\
	x_{21}^d=[100,\, 100,\, 6]^\T,\quad x_{32}^d=x_{43}^d=[0,\, -200,\, 0]^\T.
	\end{gather*}
The initial conditions are  $x_1=[-200,\, 0,\, 0]^\T$, $x_2=[-100,\, -50,\, 0]^\T$, $x_3=[0,\, 0,\, 20]^\T$, $x_4=[100,\, 100,\, -1]^\T$, $v_1=[0,\, 0,\, 0]^\T$, $v_2=[0,\, 0,\, 10]^\T$ and $v_3=v_4=[0,\, 10,\, 0]^\T$.

\begin{figure}
\centerline{
	\subfigure[initial configuration]{
		\includegraphics[width=0.8\columnwidth]{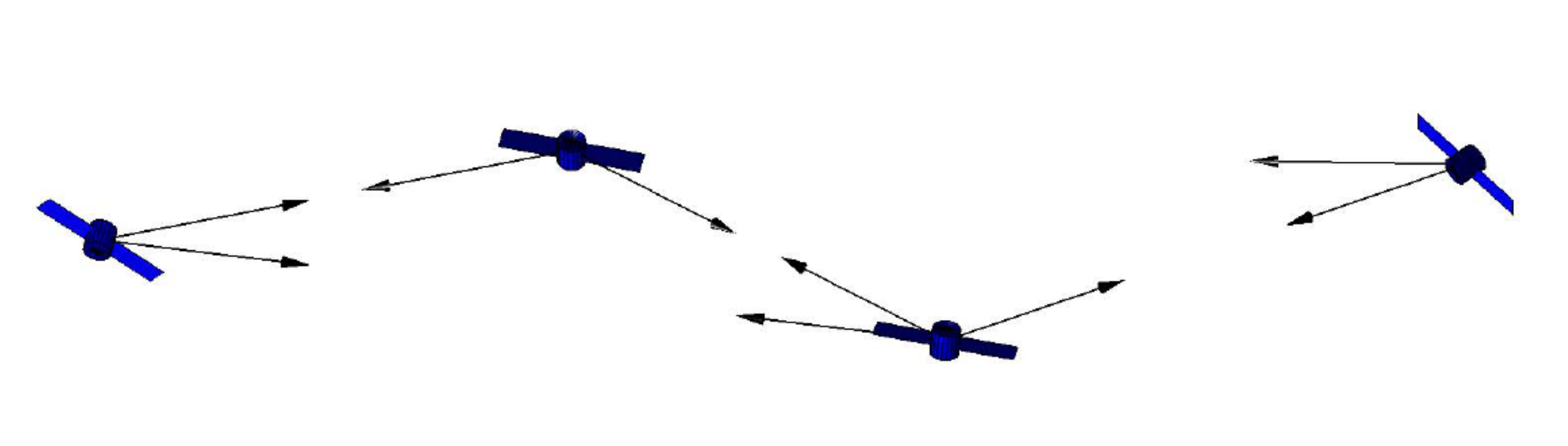}}
}
\centerline{
	\subfigure[terminal configuration]{
		\includegraphics[width=0.8\columnwidth]{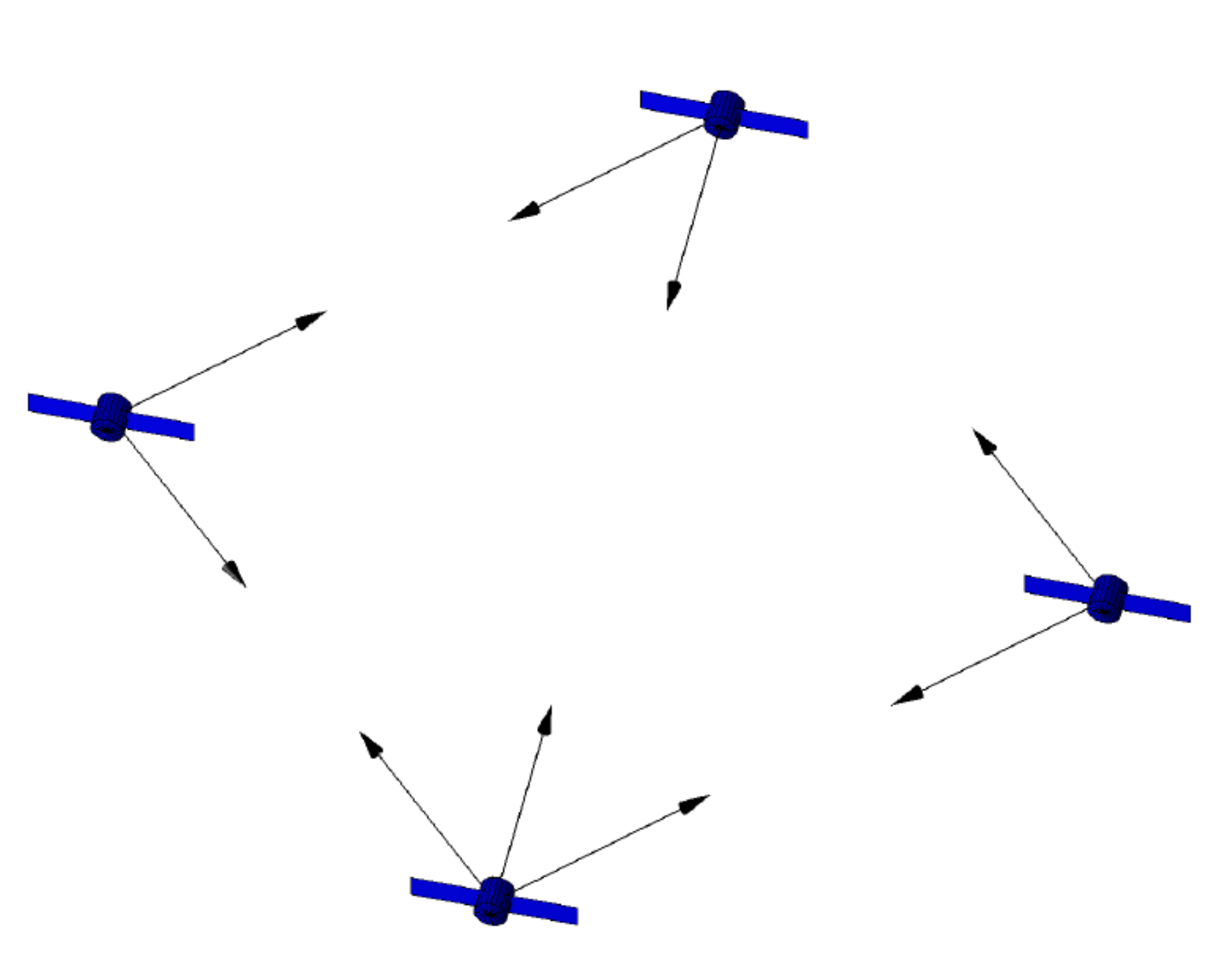}}
}
\caption{The initial configuration and the terminal configuration of four spacecraft}\label{fig:Conf}
\end{figure}
The corresponding numerical results is illustrated by Figure \ref{fig:err2}, where the attitude error vectors are defined as
		\begin{gather*} 
		e_{R_1}=\frac{1}{2}	({R_1^d}^\T R_1-R_1^\T R_1^d)^\vee\in\Re^3, \\
		e_{Q_{i,i-1}}=\frac{1}{2}({Q_{i,i-1}^d}^\T Q_{i,i-1}-Q_{i,i-1}^\T Q_{i,i-1}^d)^\vee\in\Re^3.
		\end{gather*} 
The initial configuration and the terminal configuration of spacecraft are also illustrated at Figure \ref{fig:Conf}.  

\begin{figure}
\centerline{
	\subfigure[Attitude error vector $e_R,e_{Q_{21}},e_{Q_{32}},e_{Q_{43}}$] {\includegraphics[width=0.49\columnwidth]{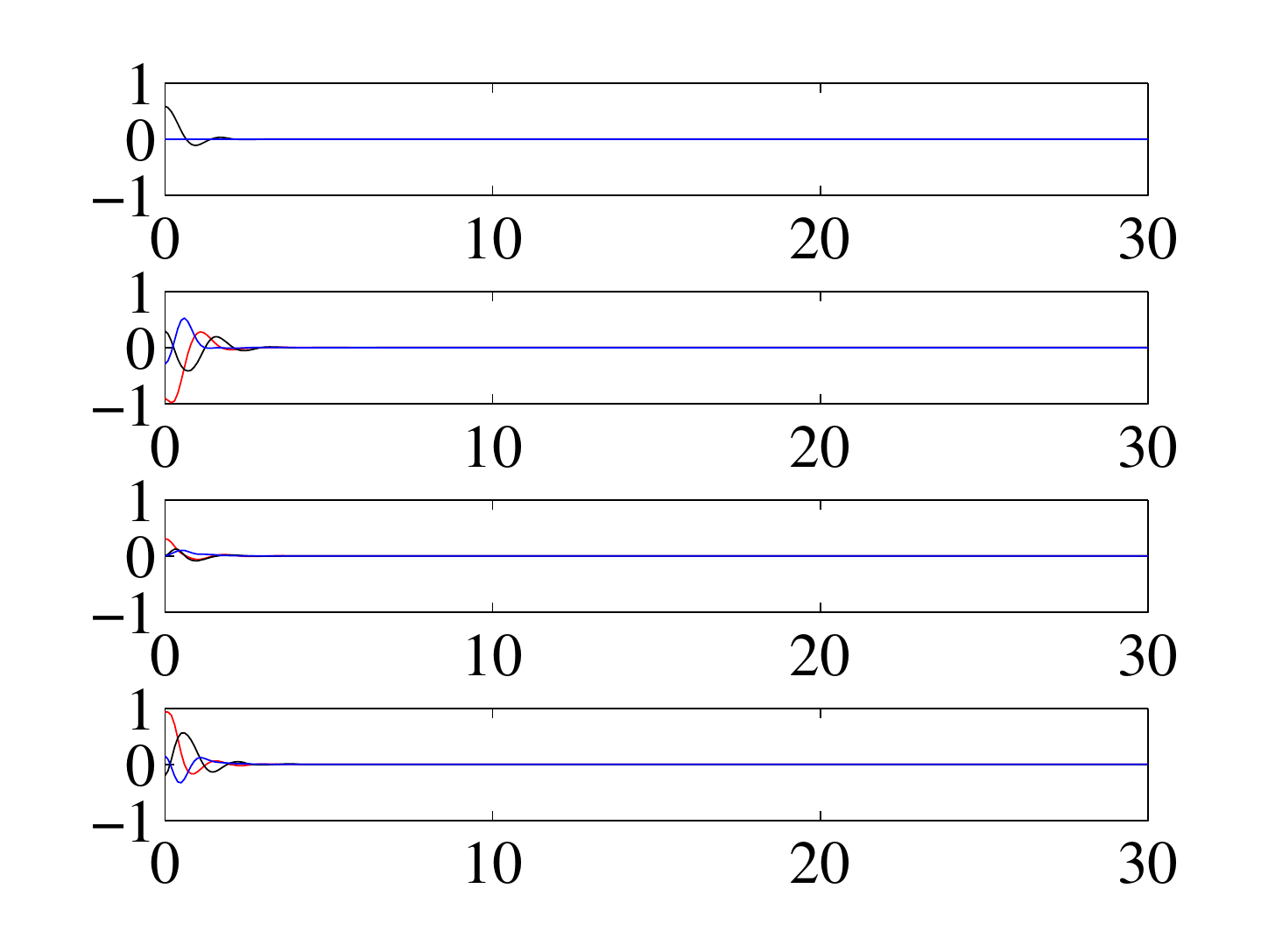}}
	\hspace*{0.001\columnwidth}
	\subfigure[Angular velocity error vector $e_{\W_1},e_{\W_2},e_{\W_3},e_{\W_4}$] {\includegraphics[width=0.49\columnwidth]{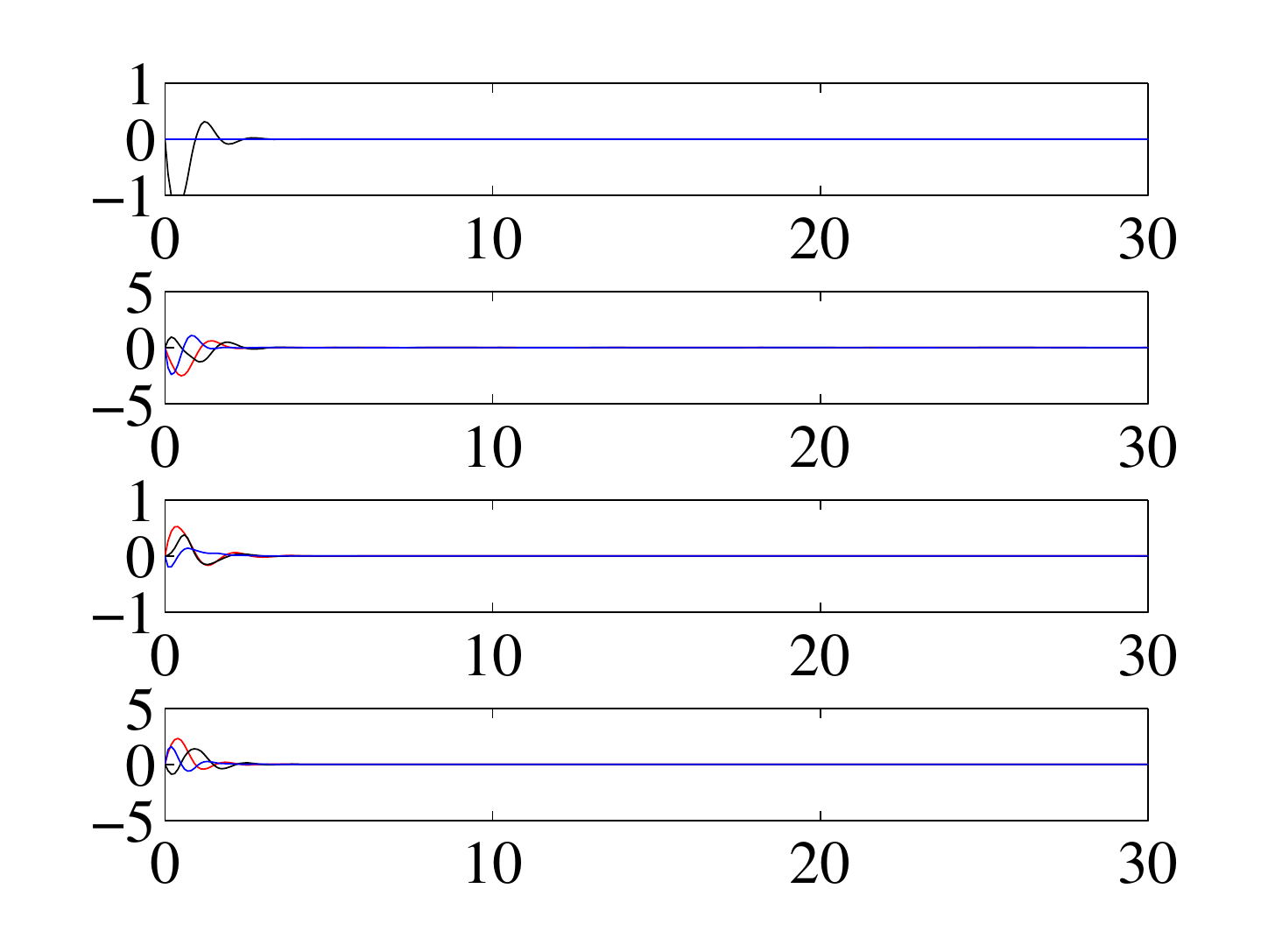}}
}
\centerline{
	\subfigure[Position error vector $e_{x_1},e_{x_{21}},e_{x_{32}},e_{x_{43}}$] {\includegraphics[width=0.47\columnwidth]{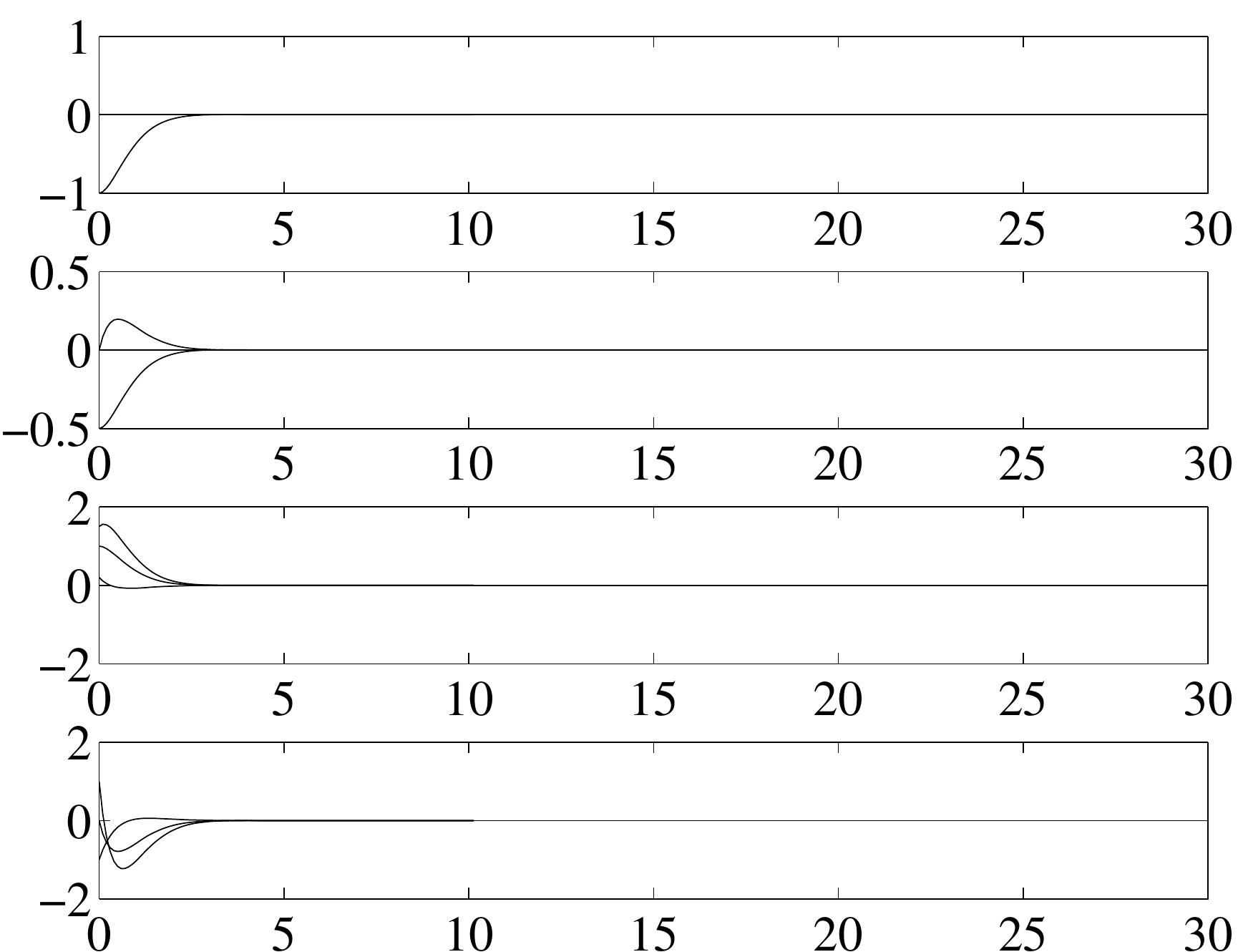}}
	\hspace*{0.03\columnwidth}
	\subfigure[Velocity error vector $e_{v_1},e_{v_{21}},e_{v_{32}},e_{v_{43}}$] {\includegraphics[width=0.47\columnwidth]{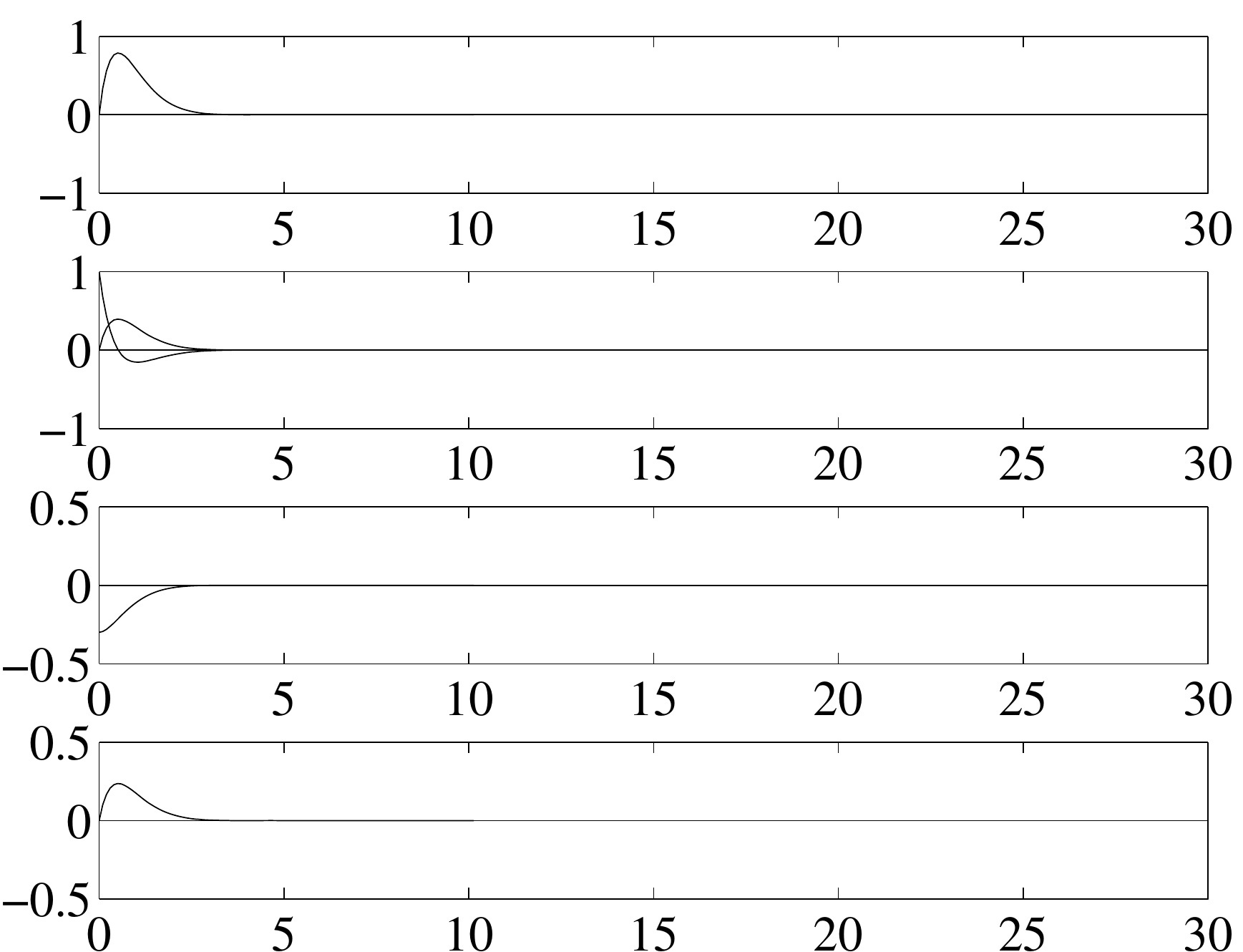}}
}
\centerline{
	\subfigure[Control moment $u_1,u_2, u_3, u_4$ (Nm)] {\includegraphics[width=0.47\columnwidth]{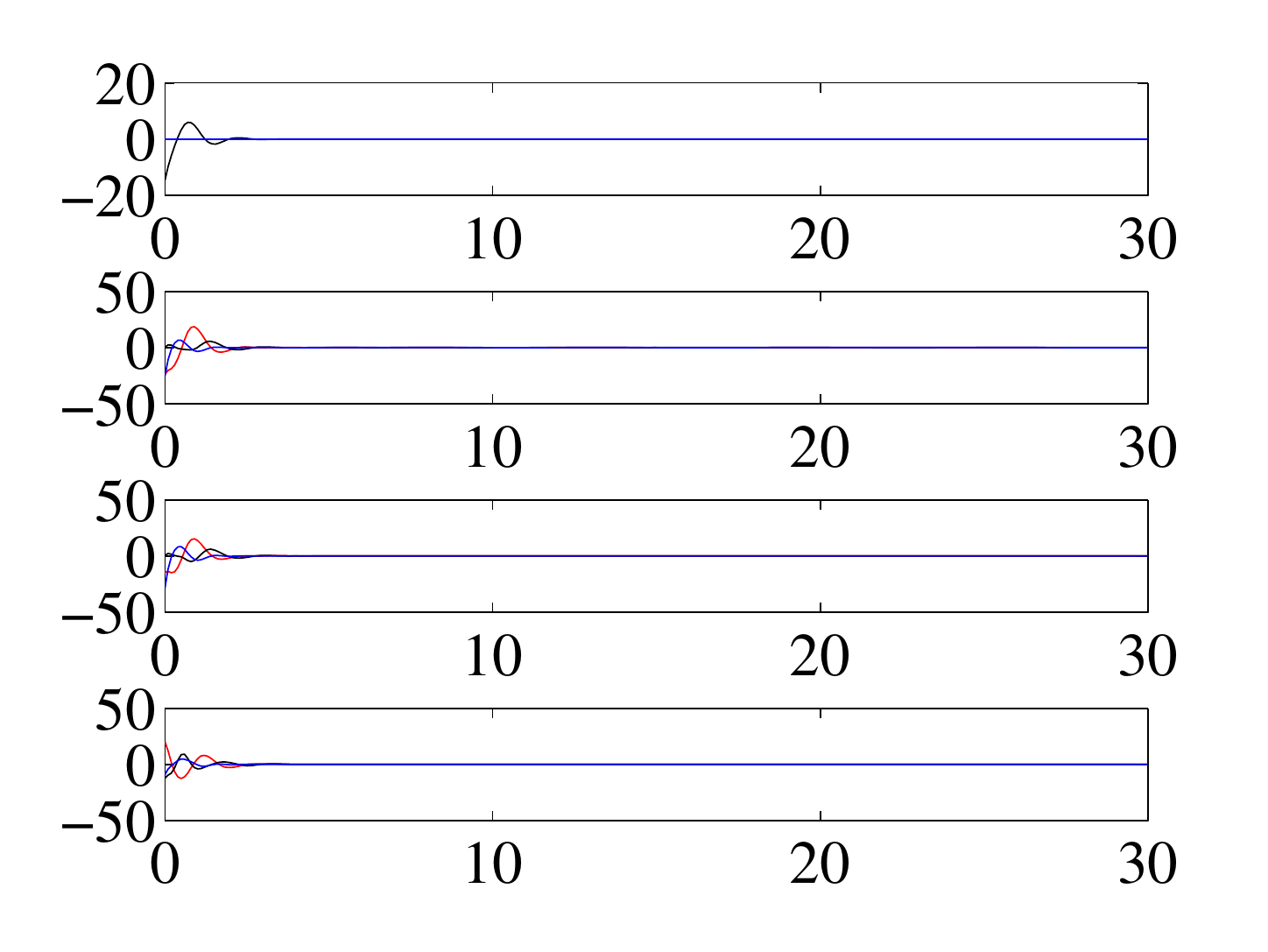}}
	\hspace*{0.03\columnwidth}
	\subfigure[Control force $f_1,f_2, f_3, f_4$ (N)] {\includegraphics[width=0.47\columnwidth]{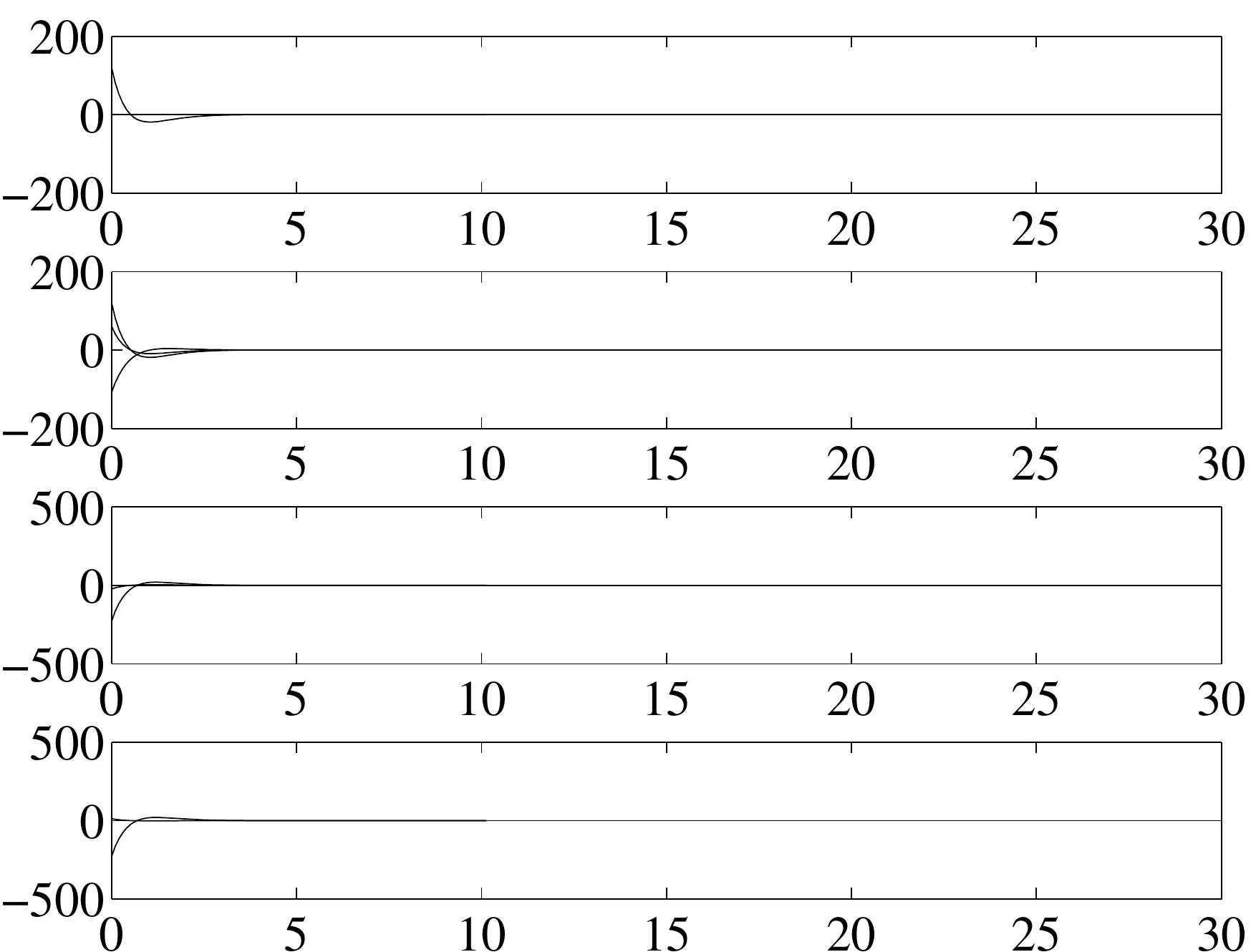}}
}
\caption{Numerical results for four spacecraft in formation.}\label{fig:err2}
\vspace*{-0.1cm} 
\end{figure}

\appendix
\subsection{\tcb{Lemmas}}
\begin{lem}\label{lem:cite1}
The properties stated in Propositon 1 of~\cite{Lee11} are summarized and extended as follows: For non-negative constants $f_1,f_2,f_3$, let $F=\text{diag}[f_1,f_2,f_3]\in\Re^{3\times 3}$, and let $P\in\SO$. Define
	\begin{gather} 
	\Phi=\frac{1}{2}\tr[F(\I-P)], \label{eq:Phi}\\
	e_P=\frac{1}{2}(FP-P^{\T}F)^\vee.\label{eq:eP}
	\end{gather}
Then, $\Phi$ is bounded by the square of the norm of $e_P$ as
	\begin{gather}
	\frac{h_1}{h_2+h_3}\|e_P\|^2\leq \Phi \leq\frac{h_1h_4}{h_5(h_1-\phi)}\|e_P\|^2. \label{eq:PHI}
	\end{gather}
If $\Phi<\phi<h_1$ for a constant $\phi$, where $h_i$ are given by
	\begin{align*} 
	h_1 &= \text{min}\{f_1+f_2,~ f_2+f_3,~ f_3+f_1\}, \\
	h_2 &= \text{max}\{(f_1-f_2)^2,~ (f_2-f_3)^2,~ (f_3-f_1)^2\}, \\
	h_3 &= \text{max}\{(f_1+f_2)^2,~ (f_2+f_3)^2,~ (f_3+f_1)^2\}, \\
	h_4 &= \text{max}\{f_1+f_2,~ f_2+f_3,~ f_3+f_1\}, \\
	h_5 &= \text{min}\{(f_1+f_2)^2,~(f_2+f_3)^2,~(f_3+f_1)^2\}.
	\end{align*}
\end{lem}
	\begin{proof}
		From Rodrigues' formula, we have $P=\exp(\hat{y})$, for $y\in\Re^3$. According to \cite{Lee11}, using the MATLAB symbolic computation tool, we find   
			\begin{align} 
			\|e_P\|^2 
			&=\frac{(1-\cos\|y\|)^2}{4\|y\|^4}\sum_{(i,j,k)\in\mathfrak{C}}(f_i-f_j)^2y_i^2y_j^2 \no\\ 
			&\quad +\frac{\sin^2\|y\|}{4\|y\|^2}\sum_{(i,j,k)\in\mathfrak{C}}(f_i+f_j)^2y_k^2, \label{eq:eR}\\
			\Phi &=\frac{1-\cos\|y\|}{2\|y\|^2} \sum_{(i,j,k)\in\mathfrak{C}}(f_i+f_j)y_k^2.\label{eq:PHI1}
			\end{align}	
		Using these two equations it follows that 
			\begin{gather*} 
	        \frac{\|e_P\|^2}{\Phi}\leq\frac{h_2+h_3}{h_1},\quad\frac{\Phi}{\|e_P\|^2} \leq\frac{h_1h_4}{h_5(h_1-\phi)},
			\end{gather*}
		which shows \refeqn{PHI}. Lemma \ref{lem:cite1} is purely the citation from~\cite{Lee11}.
	\end{proof}
	
\begin{lem}\label{lem:cite2}		
Consider 
	\begin{gather} 
	\Phi'=\frac{1}{2}\tr[F'(\I-P')], \label{eqn:Phip} \\
	e_P'=\frac{1}{2}(F'P'-P'^\T F')^\vee, \label{eqn:eP'}
	\end{gather} 
where $F'=\diag[f_1',\,f_2',\,f_3']\in\Re^{3\times 3}$, for positive scalars $f_i'$, $i=\{1,2,3\}$, and $P'=A^\T P A$, for any rotation matrix $A\in\SO$. We say that $\Phi'$ and $e_P'$ are bounded respectively in the form of 
	\begin{enumerate}[(i)]
	\item $\frac{h_1}{h_2'+h_3'}\|e_P\|^2\leq \Phi' \leq\frac{h_1h_4'}{h_5(h_1-\phi)}\|e_P\|^2$, 
	\item $\|e_{P'}\|\leq \sqrt{\frac{2h_2'+h_5'}{h_5}}\|e_P\|$,
	\end{enumerate}
where $h_i$ are given by
		\begin{align*} 
		h_2' &= \text{max}\{(f_1'-f_2')^2,~ (f_2'-f_3')^2,~ (f_3'-f_1')^2\}, \\
		h_3' &= \text{max}\{(f_1'+f_2')^2,~ (f_2'+f_3')^2,~ (f_3'+f_1')^2\}, \\
		h_4' &= \text{max}\{f_1'+f_2',~ f_2'+f_3',~ f_3'+f_1'\}, \\
		h_5' &= \text{min}\{(f_1'+f_2')^2,~(f_2'+f_3')^2,~(f_3'+f_1')^2\}.
		\end{align*}		
\end{lem}
	\begin{proof}
	 We start from finding the boundedness of $\Phi'$, employing \refeqn{exp} leads to
		\begin{gather*} 
		P'={A^\T}P\,A ={A^\T}\exp(\hat{y})\,A =\exp(\widehat{A^\T y}).
		\end{gather*}
	From Rodrigues' formula, we know $P'=\exp(\hat{y}')$ for $y'\in\Re^3$. Evidently, we have 
		\begin{gather*} 
		y'=A^\T y, \quad \|y'\|=\|y\|.
		\end{gather*}
	We therefore write
		\begin{align} 
		\Phi'
		&= \frac{(1-\cos\|y'\|)}{2\|y'\|^2}\sum_{(i,j,k)\in\mathfrak{C}}(f_i'+f_j')y_k'^2  \no\\
		&= \frac{(1-\cos\|y\|)}{2\|y\|^2}\sum_{(i,j,k)\in\mathfrak{C}}(f_i'+f_j')y_k^2. \label{eq:PHI2}
		\end{align}
	Comparing \refeqn{eR} with \refeqn{PHI2} yields 
		\begin{gather*} 
        \frac{\|e_P\|^2}{\Phi}\leq\frac{h_2'+h_3'}{h_1},\quad\frac{\Phi}{\|e_P\|^2} \leq\frac{h_1h_4'}{h_5(h_1-\phi)},
		\end{gather*}
	which shows (i). 
	
	Similar to \refeqn{eR}, we can have 
	\begin{align*} 
	\|e_{P'}\|^2 
	&=\frac{(1-\cos\|x'\|)^2}{4\|x'\|^4}\sum_{(i,j,k)\in\mathfrak{C}}(f'_i-f'_j)^2 {x'_i}^2{x'_j}^2 \\ 
	&\quad +\frac{\sin^2\|x'\|}{4\|x'\|^2}\sum_{(i,j,k)\in\mathfrak{C}}(f'_i+f_j')^2{x'_k}^2. 
	\end{align*}
	By comparing  the magnitude $e_P$ and $e_P'$	
	\begin{align*}
	\frac{\|e_{\bar{P}}\|^2}{\|e_P\|^2}
	&\leq \frac{1+\cos\|x\|}{\|x\|^2}\frac{\max\{(\bar{f}_i-\bar{f}_j)^2\}\bar{x}_i^2\bar{x}_j^2}{\min\{(f_i+f_j)^2\}x_k^2} \\
	&\quad +\frac{\max\{(\bar{f}_i+\bar{f}_j)^2\}\|\bar{x}\|^2} {\min\{(f_i+f_j)^2\}\|x\|^2} 
	= (1+\cos\|x\|)\frac{\bar{h}_2}{h_5} +\frac{\bar{h}_5}{h_5} \\
	&\leq \frac{2\bar{h}_2+\bar{h}_5}{h_5},
	\end{align*}
which leads to (ii). It shows that $P'=A^\T P A$, for any rotation matrix $A\in\SO$, we can still bound the new $\Phi'$ and $e_P'$ in terms of original $\|e_P\|$. 
	\end{proof}

\subsection{\tcb{Proof of Proposition \ref{prop:Att_Err}}}\label{subsec:pAtt}
Property (i), (ii) has been verified in~\cite{Wu2012}. Also, it has been shown that 
	\begin{align} 
	K_1=U_1G_1U_1^\T,
	\end{align}
where $U_1\in\SO$ and $G_1=\diag[k_{b_A}, k_{b_B}, 0]$ such that
	\begin{align} 
	\tr[K_1]=\tr[G_1]=k_{b_A}+k_{b_B}\triangleq \bar{k}_b. \label{eq:trace}
	\end{align}

From \refeqn{error3}, we can tell that the magnitude of $e_{b_i}$ is bounded, since $b_i$ and $b_{i_d}$ are unit vectors, for $i=\{A,B\}$. That is, $\|e_{b_i}\|\leq1$, which leads to 
	\begin{align} 
	\|e_b\|\leq k_{b_A}\|e_{b_A}\|+k_{b_B}\|e_{b_B}\|\leq k_{b_A}+k_{b_B} =\bar{k}_b.
	\end{align}	
which shows (iii).

From \refeqn{Psi}, Differentiating $\Psi_1$ gives us	
	\begin{align} 
	\dot{\Psi}_1=-\tr[K_1\dot{R}_1^e]+ \tr[\dot{K}_1(\I-R_1^e)]. \label{eq:dPsi}
	\end{align}	
From \refeqn{Rdot} and \refeqn{Tra1}, the first term at right hand side can be written as	
\begin{align*} 
	-\tr[K_1\dot{R}_1^e] 
	&=-\tr[\hat{\W}_1R_1^d{^\T}KR_1] +\tr[\hat{\W}_1^dR_1^d{^\T}KR_1].
	\end{align*}
Applying one of the hat map properties \refeqn{hatxA} leads us to	
	\begin{align} 
	-\tr[K_1\dot{R}_1^e] 
	&=\W_1^\T(R_1^d{^\T}K_1R_1-{R_1^\T}K_1R_1^d)^\vee  \no\\
	&\quad -\W_1^d{^\T}(R_1^d{^\T}K_1R_1-{R_1^\T}K_1R_1^d)^\vee \no\\
	&=(\W_1-\W_1^d)^\T e_b =e_{\W_1}^\T e_b. \label{eq:dPsi1}
	\end{align}	
Next, we can write
	\begin{align}
	\hspace*{-0.1cm}\tr[\dot{K}_1(\I-R_1^e)] 
	&=\tr[\tU_1\tG_1\tU_1^\T(I-R_1^e)] \no\\
	&=\tr[\tG_1]-\tr[\tG_1\tU_1^\T R_1^e\tU_1] \no\\
	&=\tr\big[\tG_1(\I-\tU_1^\T R_1^e\tU_1)\big] \triangleq\tilde{\Phi}_1,
	\end{align}
where $\dot{K}_1\triangleq\tU_1\tG_1\tU_1^\T$ is symmetric. We have $\tU_1\in\SO$ and $\tG_1=\diag[\tg_1,\tg_2,\tg_3]\in\Re^{3\t3}$ where $\tg_i$ are positive constants. Then, Property (ii) of Lemma \ref{lem:cite2} is applied here. That is, we let $F'=2\tG_1$ and $P'=\tU_1^\T R_1^e\tU_1$ to obtain
	\begin{gather}
	\frac{h_1}{\tih_2+\tih_3}\|e_b\|^2\leq \tilde\Phi_1 \leq\frac{h_1\tih_4}{h_5(h_1-\tilde\phi_1)}\|e_b\|^2. 
	\label{eq:dPsi2}
	\end{gather}
If $\tilde\Phi_1<\tilde\phi_1<\tih_1$ for a constant $\phi_1$, where $\tih_i$ are given by
	\begin{align*} 
	\tih_2 &= 4\text{max}\{(\tg_1-\tg_2)^2,~ (\tg_2-\tg_3)^2,~ (\tg_3-\tg_1)^2\}, \\
	\tih_3 &= 4\text{max}\{(\tg_1+\tg_2)^2,~ (\tg_2+\tg_3)^2,~ (\tg_3+\tg_1)^2\}, \\
	\tih_4 &= 2\text{max}\{\tg_1+\tg_2,~ \tg_2+\tg_3,~ \tg_3+\tg_1\}. 
	\end{align*}
Summation of \refeqn{dPsi1} and \refeqn{dPsi2} allows us to write 	
		\begin{align*} 
		\dot{\Psi}_1\leq e_{\W_1}^\T e_b + \Gamma_1\|e_b\|,
		\end{align*}
where $\Gamma_1=\frac{h_1\tih_4}{h_5(h_1-\tilde\phi_1)}$. This shows (iv).

Differentiating \refeqn{eb},  we are able to write  
	\begin{align*} 
	\dot{\hat{e}}_b &= \hat{e}_b^\zeta+ \hat{e}_b^\xi, \\
	\hat{e}_b^\zeta &=\dot{R}_1^d{^\T}K_1R_1 +R_1^d{^\T}K_1\dot{R}_1 -\dot{R}_1^\T K_1R_1^d -R_1^\T K_1\dot{R}_1^d, \\
	\hat{e}_b^\xi &=R_1^d{^\T}\dot{K}_1R_1 -R_1^\T\dot{K}_1R_1^d.
	\end{align*} 
Using kinematics equations of the spacecraft, we know that
	\begin{align*} 
	\hat{e}_b^\zeta 
	&= (\hat{\W}_1R_1^\T K_1R_1^d+{R_1^d}^\T K_1R_1\hat{\W}_1)  \\
	&\quad -(\hat{\W}_1^d{R_1^d}^\T K_1R_1 +R_1^\T K_1R_1^d\hat{\W}_1^d).
	\end{align*}
Inserting $\W_1=e_{\W_1}+\W_1^d$ and rearrangement lead us to
	\begin{align} 
	\hat{e}_b^\zeta 
	&= (\hat{e}_{\W_1}{R_1^\T}K_1R_1^d+{R_1^d}^\T K_1R_1\hat{e}_{\W_1}) +\hat{e}_b\hat{\W}_1^d -\hat{\W}_1^d\hat{e}_b.  \label{eq:ez}
	\end{align}
Further, in view of \refeqn{xAAx} and \refeqn{hatxy}, we obtain 
	\begin{align} 
	e_b^\zeta =E_{\W_1} e_{\W_1} +e_b\t\W_1^d,
	\end{align}
where $E_{\W_1}=\{\tr[{R_1^\T}K_1R_1^d]\I-{R_1^\T}K_1R_1^d\}\in\Re^{3\t3}$. In particular, the Frobenius norm of $E_{\W_1}$ is given by
	\begin{align*}
	\|E_{\W_1}\|_F 
	&=\sqrt{\tr[E_{\W_1}^\T E_{\W_1}]}=\sqrt{\tr[{R_1^\T}K_1R_1^d]^2-\tr[K_1^2]}  \\
	&\leq \frac{1}{\sqrt{2}}\tr[G_1] =\frac{1}{\sqrt{2}}\bar{k}_b,
	\end{align*}
where \refeqn{trace} is applied. From \refeqn{Bd}, we can write 
	\begin{align} 
	\|e_b^\zeta\| \leq \frac{1}{\sqrt{2}}\bar{k}_b \|e_{\W_1}\| +\Bd\|e_b\|. \label{eq:ebn}
	\end{align}

As for $\hat{e}_b^\xi$, we know that
	\begin{align*} 
	\hat{e}_b^\xi &=R_1^d{^\T}(\tU_1\tG_1\tU_1^\T)R_1 -R_1^\T(\tU_1\tG_1\tU_1^\T)R_1^d,
	\end{align*}
and this yields
	\begin{align*} 
	\|\hat{e}_b^\xi\| 
	&= \|\tU_1^\T R_1^d(\hat{e}_b^\xi){R_1^d}^\T\tU_1\| \\
	&= \|\tG_1{\tU_1^\T}R_1{R_1^d}^\T\tU_1 -\tU_1^\T R_1^dR_1^\T\tU_1\tG_1\| \\
	&\triangleq \|\tG_1\breve{U}_1 -\breve{U}_1^\T\tG_1\|,
	\end{align*}
where $\breve{U}_1={\tU_1^\T}R_1{R_1^d}^\T\tU_1=\tU_1^\T R_1^e\tU_1$. Notice that we have applied $\|A\|=\|RA\|=\|AR\|$, for any matrix $A\in\Re^{3\t3}$ and $R\in\SO$. Also we know the vee map does not change the magnitude of the vector, which implies 
	\begin{align*} 
	\|\tG_1\breve{U}_1-\breve{U}_1^\T\tG_1\| =\|(\tG_1\breve{U}_1-\breve{U}_1^\T\tG_1)^\vee\| =	\|e_b^\xi\|. 
	\end{align*}
Using property (ii) of Lemma \ref{lem:cite2}, we know that 
	\begin{align} 
	\|\hat{e}_b^\xi\| \leq B_1\|e_b\|, \label{eq:eb2}
	\end{align}
where $B_1=2\sqrt{\frac{2\tilde{h}_2+\tilde{h}_5}{h_5}}$. The summation of \refeqn{ebn} and \refeqn{eb2} results in
	\begin{align} 
	\|\dot{e}_b\| \leq \frac{1}{\sqrt{2}}\bar{k}_b \|e_{\W_1}\| +(\Bd+B_1)\|e_b\|,
	\end{align}			
which shows (v).	
	
\subsection{\tcb{Proof of Proposition \ref{prop:Rel_Err}}}\label{subsec:pRel}
The configuration error function \refeqn{Psi21} can be written as
	\begin{align*} 
	\Psi_{21}  
	&= k_{21}^\al[1+b_{12}\cdot(Q_{21}^d b_{21})] +k_{21}^\be[1+b_{123}\cdot(Q_{21}^d b_{213})] \\
	&= k_{21}^\al +k_{21}^\be +\tr[k_{21}^\al Q_{21}^d b_{21}b_{12}^\T +k_{21}^\be Q_{21}^d b_{213}b_{123}^\T] \\ 
	&= k_{21}^\al +k_{21}^\be +\tr[k_{21}^\al Q_{21}^d(R_2^\T s_{21})(s_{12}^\T{R_1})  \\
    &\quad +k_{21}^\be Q_{21}^d(R_2^\T s_{213})(s_{123}^\T{R_1})], 
	\end{align*}    
employing the geometric constraints $s_{12}=-s_{21}$ and $s_{123}=-s_{213}$, we obtain
	\begin{align} 
	\Psi_{21}     
	&= k_{21}^\al +k_{21}^\be 
	-\tr[ Q_{21}^d R_2^\T (k_{21}^\al s_{21}s_{21}^\T +k_{21}^\be s_{213}s_{213}^\T)R_1] \no\\
	&\triangleq k_{21}^\al +k_{21}^\be -\tr[Q_{21}^d R_2^\T K_{21}R_1], \label{eqn:PsiIJ1}
	\end{align}
where we define a symmetric matrix, 
	\begin{align}
	K_{21}=k_{21}^\al s_{21}s_{21}^\T +k_{21}^\be s_{213}s_{213}^\T. \label{eqn:K12}
	\end{align}

According to the spectral theory, the matrix $K_{12}$ can be decomposed into $K_{12}=U_{12}G_{12}U_{12}^T$, where $G_{12}$ is the diagonal matrix given by $G_{12}=\diag[k^\al_{12},\,k^\be_{12},\,0]\in\Re^3$, and $U_{12}$ is an orthonormal matrix defined as $U_{12}=[s_{12}, s_{123}, \frac{s_{12}\t s_{123}}{\|s_{12}\t s_{123}\|} ]\in\SO$. Hence,  the substitution of $k_{12}^\al +k_{12}^\be=\tr[G_{12}]=\tr[K_{12}]$ leads us to 
	\begin{align} 
	\Psi_{21}
	&= \tr[K_{21}(\I-R_1Q_{21}^dR_2^\T)] \label{eq:PsiIJ1}\\
	&= \tr[G_{21}(\I-U_{21}^\T R_1 Q_{21}^d R_2^\T U_{21})], \label{eq:PsiIJ2}
	\end{align}
These are alternative expressions of $\Psi_{21}$. In addition, \refeqn{e21} can be written as  
	\begin{align} 
	\hat{e}_{21}
	&= k_{21}^\al[({Q_{21}^d}^\T b_{12})\t b_{21}]^\wedge +k_{21}^\be[({Q_{21}^d}^\T b_{123})\t b_{213}]^\wedge \no\\
&= k_{21}^\al[b_{21}b_{12}^\T{Q_{21}^d} -{Q_{21}^d}^\T b_{12}b_{21}^\T] \no\\ 
&\quad +k_{21}^\be[b_{213}b_{123}^\T{Q_{21}^d}-{Q_{21}^d}^\T b_{123}b_{213}^\T] \no\\
	&=k_{21}^\al[R_2^\T s_{21}s_{12}^\T R_1{Q_{21}^d} -{Q_{21}^d}^\T{R_1^\T}s_{12} s_{21}^\T R_2] \no\\ 
	&\quad +k_{21}^\be[R_2^\T s_{213}s_{123}^\T{R_1}{Q_{21}^d}-{Q_{21}^d}^\T{R_1^\T}s_{123}s_{213}^\T R_2] \no\\
&= {Q_{21}^d}^\T{R_1^\T}K_{21}R_2 -R_2^\T K_{21}{R_1}{Q_{21}^d}, \label{eq:eij}
	\end{align} 
where one of the hat properties \refeqn{hatxy} is applied.  Together with \refeqn{PsiIJ1},  property (i) is proved.

We apply Lemma \ref{lem:cite1} here, that is, in view of \refeqn{PsiIJ2} and \refeqn{Phi}, let  $F=2G_{12}$ and $P=U_{21}^\T R_1 Q_{21}^d R_2^\T U_{21}$, we obtain $\Psi_{12}=\Phi$ and $\psi=\phi$, which results in
	\begin{gather}
	\frac{h_1}{h_2+h_3}\|e_P\|^2\leq \Psi_{12} \leq\frac{h_1h_4}{h_5(h_1-\phi)}\|e_P\|^2. 
	\label{eqn:c}
	\end{gather}
Furthermore, substituting the new expression of $F$ and $P$ into \refeqn{eP}, we obtain
	\begin{align*}
	\hat e_P 
	&= G_{21}U_{21}^\T R_1{Q_{21}^d} R_2^\T U_{21}- U_{21}^\T R_2{Q_{21}^d}^\T R_1^\T U_{21} G_{21}\\
	&=U_{21}^\T (K_{21} R_1{Q_{21}^d}R_2^\T -  R_1{Q_{21}^d}^\T{R_1^\T}K_{21})U_{21}\\
	&=U_{21}^\T R_2{Q_{21}^d}^\T(\hat{e}_{21}){Q_{21}^d} R_2^T U_{21}. 
	\end{align*}
Employing \refeqn{RxR} yields to
	\begin{align*}
	\hat e_P 
	& =(U_{21}^\T R_2{Q_{21}^d}^\T\hat{e}_{21})^\wedge.
	\end{align*}
This implies $\|e_P\|=\|e_{21}\|$ since $U_{21}$, $R_2$, $Q_{21}^d$  all belongs to $\SO$. This shows (ii).

The time-derivative of \refeqn{PsiIJ1} is given by 
	\begin{align} 
	\dot\Psi_{21}     
	&=\tr[\dot{K}_{21}(\I-{Q_{21}^e})] -\tr[K_{21}\,\dot{Q}_{21}^e], \label{eq:dPsiIJ1}
	\end{align}	
where ${Q_{21}^e}=R_1 Q_{21}^d R_2^\T\in\SO$. Notice that $\dot{K}_{12}$ is symmetric, since
	\begin{align*} 
	\dot{K}_{21}
	&= k_{21}^\al\hat{\mu}_{21}s_{21}s_{21}^\T -k_{21}^{\al}s_{21}s_{21}^\T\hat{\mu}_{21} \\ &\quad +k_{21}^\beta\hat{\mu}_{213}s_{213}s_{213}^\T -k_{21}^{\beta}s_{213}s_{213}^\T\hat{\mu}_{213} =\dot{K}_{21}^\T, 
	\end{align*}
and it can be decomposed to $\dot{K}_{21}=\tU_{21}\tG_{21}\tU_{21}^\T$ where $\tU_{21}\in\SO$ and $\tG_{21}=\diag[\tg_1,\,\tg_2,\,\tg_3]$ for three positive scalars. Hence, the first term on the right of \refeqn{dPsiIJ1} can be written as 
	\begin{align*} 
	\tr[\dot{K}_{21}(\I-{Q_{21}^e})]=\tr[\tG_{21}(\I- \tU_{21}^\T{Q_{21}^e}\tU_{21})] \triangleq \tcb{\tilde{\Phi}_{21}} .
	\end{align*}
	
Here we apply the property (ii) of Lemma. That is, we let $F'=2\tG_{21}$ and $P'=\tU_{21}^\T{Q_{21}^e}\tU_{21}$ to obtain
	\begin{gather}
	\frac{h_1}{\tih_2+\tih_3}\|e_{21}\|^2\leq \tcb{\tilde{\Phi}_{21}} \leq\frac{h_1\tih_4}{h_5(h_1-\psi)}\|e_{21}\|^2. 
	\label{eq:dP1}
\end{gather}
If $\tilde\Phi<\tilde\phi<\tih_1$ for a constant $\phi$, where $\tih_i$ are given by
\begin{align*} 
\tih_2 &= 4\text{max}\{(\tg_1-\tg_2)^2,~ (\tg_2-\tg_3)^2,~ (\tg_3-\tg_1)^2\}, \\
\tih_3 &= 4\text{max}\{(\tg_1+\tg_2)^2,~ (\tg_2+\tg_3)^2,~ (\tg_3+\tg_1)^2\}, \\
\tih_4 &= 2\text{max}\{\tg_1+\tg_2,~ \tg_2+\tg_3,~ \tg_3+\tg_1\}. 
\end{align*}
In addition, to find out the second term on the right of \refeqn{dPsiIJ1}, we first find out the time derivative of $Q_{12}^e$,
	\begin{align*}
	\dot{Q}_{21}^e
	&=\frac{d}{dt}(R_1{Q_{21}^d} R_2^\T) \\
	&= (R_1\hat{\W}_1){Q_{21}^d} R_2^\T +R_1\tcb{({Q_{21}^d}\hat{\W}_2^d -\hat{\W}_1{Q_{21}^d})}R_2^\T \\
	&\quad +R_1{Q_{21}^d}(-\hat\W_2R_2^\T)\\
	&= -R_1Q_{21}^d(\hat{\W}_2-\hat{\W}_2^d)R_2^\T = -R_1Q_{21}^d(\hat{e}_{\W_2})R_2^\T,
	\end{align*}
where \refeqn{Wijd} and \refeqn{RxR} are applied. Then we can write
	\begin{align}
	-\tr[K_{21}\dot{Q}_{21}^e] 
	&= \tr\big[K_{21}R_1Q_{21}^d(\hat{e}_{\W_2})R_2^\T\big]  \no\\
	&= \tr\big[(\hat{e}_{\W_2})R_2^\T K_{21}R_1Q_{21}^d\big]  \no\\
	&= e_{\W_2}^\T({Q_{21}^d}^\T R_1^\T K_{21}R_2-R_2^\T K_{12}R_1{Q_{21}^d})^\vee \no\\
	&= e_{\W_2}^\T e_{21},\label{eq:dP2}
	\end{align}
where  \refeqn{hatxA} and \refeqn{eij} are applied in the process of rearrangement. Next, substituting \refeqn{dP1} and \refeqn{dP2} to \refeqn{dPsiIJ1} results in 
	\begin{align} 
	\dot\Psi_{12}     
	&\leq e_{\W_2}^\T e_{21} +\Gamma_{21}\|e_{21}\|^2, \label{eqn:dPsiIJ2}
	\end{align}	
where $\Gamma_{21}\triangleq \sqrt{\frac{h_1\tih_4}{h_5(h_1-\phi)}}$. This is the proof of (iii).	

Lastly we show (iv). Differentiating \refeqn{eij} with respect to time, we obtain
	\begin{align}
	\dot{\hat{e}}_{21}=\mathfrak{\hat{e}_a}+\mathfrak{\hat{e}_b}, \label{eq:deji}
	\end{align}
where 
	\begin{align*}
	\mathfrak{\hat{e}_a} 
	&=\dot{Q}^d_{21}{}^\T R_1^\T K_{21}R_2 +{Q_{21}^d}^\T\dot{R}_1^\T K_{21}R_2 +{Q_{21}^d}^\T R_1^\T K_{21}\dot{R}_2 \\ 
	&\quad -\dot{R}_2^\T K_{21}R_1Q_{21}^d -R_2^\T K_{21}\dot{R}_1Q_{21}^d -R_2^\T K_{21}R_1\dot{Q}^d_{21}, \\
	\mathfrak{\hat{e}_b} 
	&={Q_{21}^d}^\T{R_1^\T}\dot{K}_{21}R_2 -R_2^\T\dot{K}_{21}{R_1}Q_{21}^d.
	\end{align*}	
Employing the kinematic equations \refeqn{Rdot}, \refeqn{Wijd} and one of the hat map properties \refeqn{RxR} into $\mathfrak{\hat{e}_a}$ leads to  	
	
	\begin{align*}
	\mathfrak{\hat{e}_a}
	&=-(\hat{\W}_2^d{Q_{21}^d}^\T R_1^\T K_{21}R_2 +R_2^\T K_{21}R_1Q_{21}^d\hat{\W}_2^d) \\
	&\quad+ (\hat{\W}_2R_2^\T K_{21}R_1Q_{21}^d +{Q_{21}^d}^\T{R_1^\T}K_{21}{R}_2\hat{\W}_2),
	\end{align*}	
and the substitution of $e_{\W_2}=\W_2-\W_2^d$ allows us to write	
	\begin{align*}
	\mathfrak{\hat{e}_a}
	&=({Q_{21}^d}^\T{R_1^\T}K_{21}{R}_2- R_2^\T K_{21}R_1Q_{21}^d)\hat{\W}_2^d \\ 
	&\quad -\hat{\W}_2^d({Q_{21}^d}^\T{R_1^\T}K_{21}R_2-R_2^\T K_{21}R_1Q_{21}^d) \\
	&\quad +(\hat{e}_{\W_2}R_2^\T K_{21}R_1Q_{21}^d +{Q_{21}^d}^\T{R_1^\T}K_{21}{R}_2\hat{e}_{\W_2})\\
	&= \hat{e}_{21}\hat{\W}_2^d-\hat{\W}_2^d\hat{e}_{21} \\
	&\quad +(\hat{e}_{\W_2}R_2^\T K_{21}R_1Q_{21}^d +{Q_{12}^d}^\T{R_1^\T}K_{21}{R}_2\hat{e}_{\W_2}).
	\end{align*}
Apply \refeqn{hatxy} and \refeqn{xAAx} here. Then removing the hat map on both side of the equation, we obtain 
	\begin{align}
	\mathfrak{e_a}
	&=e_{21}\t\W_2^d + \\
	&\quad \{\tr[R_2^\T K_{21}R_1Q_{21}^d]\I-R_2^\T K_{21}R_1Q_{12}^d\}e_{\W_2} \no\\
	&\triangleq e_{21}\t\W_2^d +E_{\W_2}e_{\W_j}. \label{eq:neweji}
	\end{align}		
where $E_{\W_2}=\{\tr[R_2^\T K_{21}R_1Q_{21}^d]\I-R_2^\T K_{21}R_1Q_{21}^d\}$. Moreover, the Frobenius norm of $E_{\W_2}$ is given by  
	\begin{align}
	\|E_{\W_2}\|_F
	&=\sqrt{\tr[E_{\W_2}^\T E_{\W_2}]} \no\\
	&=\frac{1}{2}\sqrt{\tr[R_2^\T K_{21}R_1Q_{21}^d]^2 +\tr[K_{21}^2]}.  \label{eq:EW}
	\end{align}
Using the fact that $K_{21}=U_{21}G_{21}U^\T_{21}$, we find
	\begin{align*}
	\tr[R_2^\T K_{12}R_1Q_{21}^d]
	&= \tr[R_2^\T U_{21}G_{21}U^\T_{21}R_1Q_{21}^d] \no\\
	&= \tr[G_{21}U^\T_{21}R_1Q_{21}^dR_2^\T U_{21}].
	\end{align*} 	
Let $U^\T_{21}R_1Q_{21}^dR_2^\T U_{21}=\exp(z)\in\SO$ from Rodrigues' formula. Using the MATLAB symbolic tool, we find
	\begin{align*}
	&\tr[U^\T_{12}R_1Q_{12}^d{}^\T R_2^\T U_{12}G_{12}] \\ &=\cos\|z\|\sum_{i=1}^3g_i(1-\frac{z_i}{\|z\|^2}) +\sum_{i=1}^3g_i\frac{x_i^2}{\|x\|^2} \leq \sum_{i=1}^3g_i=\tr[G_{21}],
	\end{align*} 
since $0\leq\frac{z_i^2}{\|z\|^2}\leq1$. Inserting this into \refeqn{EW} and recall that $\tr[K_{21}]=\tr[G_{21}]=k_{21}^\al+k_{21}^\be$, we get	
	\begin{align*}
	\|E_{\W_2}\|\leq \frac{1}{\sqrt{2}}\tr[G_{21}]= \frac{1}{\sqrt{2}}(k_{21}^\al+k_{21}^\be).
	\end{align*}	
Substituting this into \refeqn{neweji} leads to	
	\begin{align}
	\|\mathfrak{e_a}\|
	&\leq \frac{1}{\sqrt{2}}(k_{21}^\al+k_{21}^\be)\|e_{\W_2}\| +\Bd\|e_{21}\|. \label{eq:nea}
	\end{align}	
	
As for $\mathfrak{e_b}$, we can rewrite it as follows	
	\begin{align*}
		\mathfrak{\hat{e}_b} 
		&={Q_{21}^d}^\T{R_1^\T}\dot{K}_{21}R_2 -R_2^\T\dot{K}_{21}{R_1}Q_{21}^d,\\
	\|\mathfrak{\hat{e}_b}\| 
	&= \|({Q_{21}^d}^\T{R_1^\T}\tU_{21}\tG_{21}\tU_{21}^\T R_2 -R_2^\T\tU_{21}\tG_{21}\tU_{21}^\T{R_1}Q_{21}^d\|,
	\end{align*}	
with further rearrangement, we get
	\begin{align*}
	\|\mathfrak{\hat{e}_b}\| 
	&= \| \tU_{21}^\T R_1Q_{21}^d({Q_{12}^d}^\T R_1^\T\tU_{21}\,\tcb{\tG_{21}}\,\tU_{21}^\T R_2 \\ 
	&\quad -R_2^\T\tU_{21}\tcb{\tG_{21}}\,\tU_{21}^\T R_1Q_{21}^d){Q_{12}^d}^\T{R_1^\T}\tU_{21} \|\\
	&= \|(\tcb{\tG}_{21}\tU_{21}^\T R_2{Q_{21}^d}^\T R_1^\T\tU_{21} -\tU_{21}^\T R_1Q_{21}^d R_2^\T\tU_{21}\tcb{\tG_{21}})\| \\
	&\triangleq \|(\tG_{21}\breve{U}_{21} -\breve{U}_{21}^\T\tG_{21})\|
	=\|(\tcb{\tG}_{21}\breve{U}_{21} -\breve{U}_{21}^\T\tcb{\tG_{21}})^\vee\| ,
	\end{align*}	
where we have applied $\|A\|=\|RA\|=\|AR\|$, for any matrix $A\in\Re^{3\t3}$ and $R\in\SO$ and we also let $\breve{U}_{12}=\tU^\T R_2Q_{12}^d R_1^\T\tU\in\SO$. 

Use the property (ii) of Lemma \ref{lem:cite2} that we have developed, that is, let $\mathfrak{e_b}=2e_P'$, $\tilde{G}_{21}=F'$ and $\breve{U}_{21}=P'$. We claim that  
	\begin{align}
	\|\mathfrak{e_b}\|\leq 2\sqrt{\frac{2\tilde{h}_2+\tilde{h}_5}{h_5}}\|e_{21}\| 
	\triangleq B_{21}\|e_{21}\|. \label{eq:eb1}
	\end{align}	
Next, inserting \refeqn{nea}, \refeqn{eb1} into \refeqn{deji} gives rise to
	\begin{align}
	\|\dot{e}_{21}\|
	&\leq \|\mathfrak{e_a}\| +\|\mathfrak{e_b}\| \no\\
	&= \frac{1}{\sqrt{2}}(k_{12}^\al+k_{12}^\be)\|e_{\W_2}\| +(\Bd+B_{21})\|e_{21}\|, \label{eqn:deji2}
	\end{align}	
which shows (iv) directly.		


\bibliography{CDC14}
\bibliographystyle{IEEEtran}

\end{document}